\documentclass[reqno,pdftex]{amsart}     
\usepackage{mathrsfs}
 
\usepackage{a4wide} 
\usepackage[T1]{fontenc}
\usepackage[latin1]{inputenc}
\usepackage{tikz}
\usetikzlibrary{shapes.misc}
\usetikzlibrary{decorations.markings}
\usepackage{pgfplots}
\usetikzlibrary{decorations.pathreplacing}
\usetikzlibrary{hobby}
\usepackage{pstricks}
\usetikzlibrary{shapes.geometric}
\usepackage{mathtools}
\usepackage{subcaption}
\usetikzlibrary{positioning,arrows}
\usepackage{enumitem}
\usetikzlibrary{calc}

\renewcommand{\(}{\left(}
\renewcommand{\)}{\right)}
\newtheorem{theo}{Theorem}
\newtheorem{prop}{Proposition}
\newtheorem{lemma}{Lemma}
\newtheorem{cor}{Corollary\!\!}

\newtheorem{ncor}{Corollary}

\theoremstyle{definition}

\newtheorem{df}{Definition}
\newtheorem{ex}{Example}

\theoremstyle{remark}

\newtheorem{rem}{Remark\!\!}
 
\newtheorem{nrem}{Remark}

\newcommand{\bth}{\begin{theo}} 
\newcommand{\eth}{\end{theo}} 
\newcommand{\bl}{\begin{lemma}} 
\newcommand{\el}{\end{lemma}} 
\newcommand{\bp}{\begin{prop}} 
\newcommand{\ep}{\end{prop}} 
\newcommand{\bdf}{\begin{df}} 
\newcommand{\edf}{\end{df}} 
\newcommand{\brem}{\begin{rem}} 
\newcommand{\erem}{\end{rem}} 
\newcommand{\bnrem}{\begin{nrem}} 
\newcommand{\enrem}{\end{nrem}} 
\newcommand{\bex}{\begin{ex}} 
\newcommand{\eex}{\end{ex}} 
\newcommand{\bcor}{\begin{cor}} 
\newcommand{\ecor}{\end{cor}} 
\newcommand{\bncor}{\begin{ncor}} 
\newcommand{\encor}{\end{ncor}} 
\newcommand{\bpf}{\begin{proof}} 
\newcommand{\epf}{\end{proof}}

\newcommand{\tend}{\longrightarrow}

\definecolor{azure(colorwheel)}{rgb}{0.0, 0.5, 1.0}
\definecolor{ao(english)}{rgb}{0.0, 0.5, 0.0}
\definecolor{lapislazuli}{rgb}{0.15, 0.38, 0.61}
\definecolor{lasallegreen}{rgb}{0.03, 0.47, 0.19}
\definecolor{darklavender}{rgb}{0.45, 0.31, 0.59}
\definecolor{persianrose}{rgb}{1.0, 0.16, 0.64}
\definecolor{teal}{rgb}{0.0, 0.5, 0.5}
\definecolor{orange(colorwheel)}{rgb}{1.0, 0.5, 0.0}
\definecolor{orange(ryb)}{rgb}{0.98, 0.6, 0.01}

\begin{document}

\title[Variables in lambda-terms with bounded De Bruijn Indices and De Bruijn Levels]{Distribution
of variables in lambda-terms with restrictions on De Bruijn indices and De Bruijn levels} 
\author{Bernhard Gittenberger \and Isabella Larcher} 
\thanks{This research has been supported by the Austrian Science Fund (FWF) grant SFB F50-03.
\\
A preliminary version of this work was presented at AofA'2018. The present paper is the full
version of \cite{GiLa18}.} 
\address{Department of Discrete Mathematics and Geometry, Technische  
Universit\"at Wien, Wiedner Hauptstra\ss e 8-10/104, A-1040 Wien, Austria.}
\email{gittenberger@dmg.tuwien.ac.at} 
\address{Department of Discrete Mathematics and Geometry, Technische
Universit\"at Wien, Wiedner Hauptstra\ss e 8-10/104, A-1040 Wien, Austria.}
\email{isabella.larcher@tuwien.ac.at}




\begin{abstract} 
We investigate the number of variables in two special subclasses of lambda-terms that are restricted by a bound of the number of abstractions between a variable and its binding lambda, the so-called De-Bruijn index, or by a bound of the nesting levels of abstractions, \textit{i.e.}, the number of De Bruijn levels, respectively. These restrictions are on the one hand very natural from a practical point of view, and on the other hand they simplify the counting problem compared to that of unrestricted lambda-terms in such a way that the common methods of analytic combinatorics are applicable.

We will show that the total number of variables is asymptotically normally distributed for both subclasses of lambda-terms with mean and variance asymptotically equal to $Cn$ and $\tilde{C}n$, respectively, where the constants $C$ and $\tilde{C}$ depend on the bound that has been imposed. So far we just derived closed formulas for the constants in case of the class of lambda-terms with bounded De Bruijn index. However, for the other class of lambda-terms that we consider, namely lambda-terms with a bounded number of De Bruijn levels, we investigate the number of variables, as well as abstractions and applications, in the different De Bruijn levels and thereby exhibit a so-called ``unary profile'' that attains a very interesting shape.
\end{abstract} 
 
\maketitle

\section{Introduction} 

Lambda-calculus is a set of rules to manipulate lambda-terms and it is an important tool in
theoretical computer science. To our knowledge, the first appearance of enumeration problems in
the sense of enumerative combinatorics which are linked to lambda-calculus is found in
\cite{Ha96}, where certain models of lambda-calculus are analyzed which have representations as
formal power series. More recently, we observe rising interest in the quantitative properties of
large random lambda-terms. The first work in this direction seems to be \cite{MTZ00}. Later David
\emph{et al.} \cite{DGKRT13} investigated the proportion of normalising terms, which was also the
topic of \cite{BGZ17} in a different context. Other papers dealing with certain structural
properties of lambda-terms are for instance \cite{BoTa18,Gr16,SAKT17}.

Since studying quantitative aspects of lambda-terms using combinatorial methods relies heavily
on their enumeration, many papers are devoted to their enumeration, which itself very much depends
on the particular class of terms and the definition of the term size. The enumeration may be done
by contructing bijections to certain classes of maps, see e.g. \cite{MR3101704,Ze16,ZeGi15} or
the use of the methodology from analytic combinatorics \cite{MR2483235}, see e.g. 
\cite{BGLZ17,MR2815481,bodini2015number,MR3158269,BGG18,GrLe15,MR3018087}. 

Another approach to gain structural insight is by random generation. Solving the enumeration
problems is the basis for an efficient algorithm for this purpose, namely Boltzmann sampling
\cite{DFLS04,FFP07}. The method is extendible to a multivariate setting allowing for a fine tuning
according to specified structural properties of the sampled objects, as was demonstrated in
\cite{BBD18,BoPo10}. The generation of lambda-terms was treated in 
\cite{BGT17,MR3101704,GrLe15,Pa11,Ta17,Wa05}.

\smallskip
In \cite{bodini2015number} the authors discovered a very interesting phenomenon concerning the generating function of lambda-terms with a bounded number of De Bruijn levels, namely that the asymptotic behaviour of the coefficients of the generating function changes with the imposed bound. More precisely, the type of the dominant singularity changes from $\frac{1}{2}$ to $\frac{1}{4}$ whenever the bound belongs to a special doubly-exponentially growing sequence. This alteration of the type of the dominant singularity has a direct impact on the polynomial factor of the coefficient's asymptotics, namely it shifts from $n^{-3/2}$ to $n^{-5/2}$ for $n$ tending to infinity. This paper studies the structure of random large lambda-terms belonging to this class and thereby delivers an explanation of the above mentioned phenomenon, since it arises from the location of the variables within the lambda-term. 

\bigskip
The lambda calculus was invented by Church and Kleene in the 1930ies as a tool for the investigation of decision problems. Today it still plays an important role in computability theory and for automatic proof systems. Furthermore, it represents the basis for some programming languages, such as LISP. For a thorough introduction to lambda calculus we refer to \cite{MR3235567}. This paper does not require any preliminary knowledge of lambda calculus in order to follow the proofs. Instead we will study the basic objects of lambda calculus, namely lambda-terms, by considering them as combinatorial objects, or more precisely as a special class of directed acyclic graphs (DAGs).

\begin{df}[{lambda-terms, \cite[Definition 3]{MR3063045}}]
\label{def:lambda-terms}
Let $\mathcal{V}$ be a countable set of variables.
The set $\Lambda$ of lambda-terms is defined by the following grammar:
\begin{enumerate}
	\item every variable in $\mathcal{V}$ is a lambda-term,
	\item if $T$ and $S$ are lambda-terms then $TS$ is a lambda-term, (application)
	\item if $T$ is a lambda-term and $x$ is a variable then $\lambda x.T$ is a lambda-term. (abstraction)
\end{enumerate}
\end{df}

The name application arises, since lambda-terms of the form $TS$ can be regarded as functions $T(S)$, where the function $T$ is applied to $S$, which in turn can be a function itself. An abstraction can be considered as a quantifier that binds the respective variable in the sub-lambda-term within its scope.
Both application and repeated abstraction are not commutative, \textit{i.e.}, in general the lambda-terms $TS$ and $ST$, as well as $\lambda x. \lambda y. M$ and $\lambda y. \lambda x. M$, are different (with the exceptions of $T=S$ and none of the variables $x$ or $y$ occurring in $M$, respectively). Each $\lambda$ binds exactly one variable (which may occur several times in the terms), and since we will just focus on a special subclass of closed lambda-terms, each variable is bound. 

We will consider lambda-terms modulo $\alpha$-equivalence, which means that we identify two lambda-terms if they only differ by the names of their bound variables. For example $\lambda x. ( \lambda y. (xy)) \equiv \lambda y. ( \lambda z. (yz))$.
1972 De Bruijn (\cite{de1972lambda}) introduced a representation for lambda-terms that completely avoids the use of variables by substituting them by natural numbers that indicate the number of abstractions between the variable and its binding lambda (the binding lambda is counted as well), \textit{i.e.}, $\lambda x. ( \lambda y. (xy)) = \lambda(\lambda21)$.

\begin{df}[De Bruijn index, De Bruijn level]
The natural numbers that represent the variables in the De Bruijn representation of a lambda-term are called De Bruijn indices. The number of nested lambdas starting from the outermost one specifies the De Bruijn level in which a variable (or De Bruijn index, respectively) is located.
\end{df}

For example in the lambda-term $\lambda x. x( \lambda y. (xy)) = \lambda1(\lambda21)$ the first occurrence of the variable $x$ (\textit{i.e.}, the leftmost 1 in the De Bruijn representation) is in the first De Bruijn level, while the other variables are in the second De Bruijn level.

There is also a combinatorial interpretation of lambda-terms that considers them as DAGs and thereby naturally identifies two $\alpha$-equivalent terms to be equal. Combinatorially, lambda-terms can be seen as rooted unary-binary trees containing additional directed edges. Note that in general the resulting structures are not trees in the sense of graph theory, but due to their close relation to trees (see Definition \ref{def:lambda-DAG}) some authors call them lambda-trees or enriched trees. We will call them lambda-DAGs in order to emphasise that these structures are in fact DAGs, if we consider the undirected edges of the underlying tree to be directed away from its root.

\begin{df}[{lambda-DAG, \cite[Definition 5]{MR3063045}}]
\label{def:lambda-DAG}
 
With every lambda-term $T$, the corresponding lambda-DAG $G(T)$ can be constructed in the following way:
\begin{enumerate}
	\item If $x$ is a variable then $G(x)$ is a single node labeled with $x$. Note that $x$ is unbound.
	\item $G(PQ)$ is a lambda-DAG with a binary node as root, having the two lambda-DAGs $G(P)$ (to the left) and $G(Q)$ (to the right) as subgraphs.
	\item The DAG G($\lambda x.P$) is obtained from $G(P)$ in four steps:
	\begin{enumerate}
		\item Add a unary node as new root.
		\item Connect the new root by an undirected edge with the root of G(P).
		\item Connect all leaves of $G(P)$ labelled with $x$ by directed edges with the new root, where the root is start vertex of these edges.
		\item Remove all labels $x$ from $G(P)$. Note that now $x$ is bound.
	\end{enumerate} 
\end{enumerate}
 
Obviously, applications correspond to binary nodes and abstractions correspond to unary nodes of the underlying Motzkin-tree that is obtained by removing all directed edges. Of course, in the lambda-DAG some of the vertices that were former unary nodes might have gained out-going edges, so they are no unary nodes in the lambda-DAG anymore. However, when we speak of unary nodes in the following, we mean the unary nodes of the underlying unary-binary tree that forms the skeleton of the lambda-DAG.
\end{df}

\begin{figure}[h]
  \centering
\scalebox{0.8}{	 \begin{minipage}[t]{.45\linewidth}
		\centering
	\begin{tikzpicture}
	\fill (0,0) circle (0.1);
	\fill (2,0) circle (0.1);
	\fill (1,1) circle (0.1);
	\fill (1,2) circle (0.1);
	\fill (2,3) circle (0.1);
	\fill (3,2) circle (0.1);
	\fill (2,4) circle (0.1);
	\draw[
    decoration={markings,mark=at position 1 with {\arrow[scale=2]{>}}},
    postaction={decorate},
    shorten >=0.4pt
    ]
    (2,4) to[bend right=20] (0,0.08);
    \draw[
    decoration={markings,mark=at position 1 with {\arrow[scale=2]{>}}},
    postaction={decorate},
    shorten >=0.4pt
    ]
    (2,4) to[bend left=20] (3,2.08);
    \draw[
    decoration={markings,mark=at position 1 with {\arrow[scale=2]{>}}},
    postaction={decorate},
    shorten >=0.4pt
    ]
    (1,2) to[bend left=20] (2.01,0.08);
	\draw (0,0) -- (1,1);
	\draw (2,0) -- (1,1);
	\draw (1,1) -- (1,2);
	\draw (1,2) -- (2,3);
	\draw (2,3) -- (3,2);
	\draw (2,3) -- (2,4);
	\end{tikzpicture}
	\vspace{2mm}
	\subcaption{$\lambda x. (( \lambda y. (xy))x)$ \\ $= \lambda((\lambda21)1)$}
	\end{minipage}%
	 \hfill%
  \begin{minipage}[t]{.5\linewidth}
	\centering
	\begin{tikzpicture} 
	\fill (0,1) circle (0.1);
	\fill (1,2) circle (0.1);
	\fill (2,1) circle (0.1);
	\fill (2,0) circle (0.1);
	\fill (3,0) circle (0.1);
	\fill (5,0) circle (0.1);
	\fill (4,1) circle (0.1);
	\fill (4,2) circle (0.1);
	\fill (1,3) circle (0.1);
	\fill (4,3) circle (0.1);
	\fill (2.5,4) circle (0.1);
	\draw[
    decoration={markings,mark=at position 1 with {\arrow[scale=2]{>}}},
    postaction={decorate},
    shorten >=0.4pt
    ]
    (1,3) to[bend right=20] (0,1.08);
    \draw[
    decoration={markings,mark=at position 1 with {\arrow[scale=2]{>}}},
    postaction={decorate},
    shorten >=0.4pt
    ]
    (2,1) to[bend left=25] (2.07,0.08);
    \draw[
    decoration={markings,mark=at position 1 with {\arrow[scale=2]{>}}},
    postaction={decorate},
    shorten >=0.4pt
    ]
    (4,2) to[bend left=20] (5.01,0.08);
    \draw[
    decoration={markings,mark=at position 1 with {\arrow[scale=2]{>}}},
    postaction={decorate},
    shorten >=0.4pt
    ]
    (4,3) to[bend right=20] (3,0.08);
	\draw (0,1) -- (1,2);
	\draw (1,2) -- (1,3);
	\draw (1,3) -- (2.5,4);
	\draw (2,0) -- (2,1);
	\draw (2,1) -- (1,2);
	\draw (2.5,4) -- (4,3);
	\draw (4,3) -- (4,2);
	\draw (4,2) -- (4,1);
	\draw (4,1) -- (3,0);
	\draw (4,1) -- (5,0);
	\end{tikzpicture}
	\vspace{2mm}
	\subcaption{$(\lambda x. (x (\lambda y.y))(\lambda x.(\lambda y. xy))$ \\ $= (\lambda(1(\lambda1)))(\lambda(\lambda21))$}
	\end{minipage}	}
\caption{The lambda-DAGs corresponding to the respective terms written below.}
\label{fig:lambdaDAGs}
\end{figure}
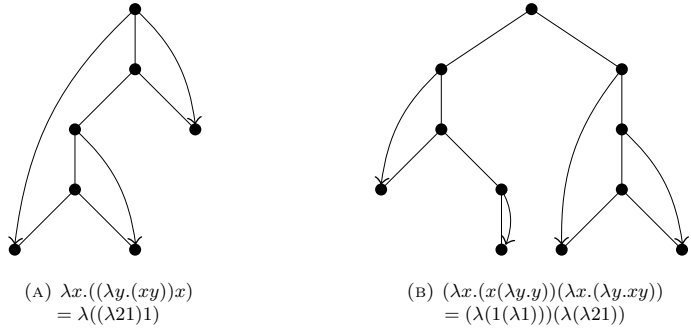

\begin{figure}[h]
\label{fig:unarylengthheight}
  \centering
\scalebox{0.8}{	 \begin{minipage}[t]{.45\linewidth}
		\centering
	\begin{tikzpicture}[scale=0.6, >=latex]
	\fill (3,8) circle (0.1);
	\fill (3,7) circle (0.1);
	\fill (0,4) circle (0.1);
	\fill (1,5) circle (0.1);
	\fill (1,6) circle (0.1);
	\fill (2,4) circle (0.1);
	\fill (5,6) circle (0.1);
	\fill (5,5) circle (0.1);
	\fill (4,4) circle (0.1);
	\fill (6,4) circle (0.1);
	\fill (3,3) circle (0.1);
	\fill (5,3) circle (0.1);
	\fill (5,2) circle (0.1);
	\fill (4,1) circle (0.1);
	\fill (6,1) circle (0.1);
	
	\draw (0,4) -- (1,5);
	\draw (1,5) -- (2,4);
	\draw (1,5) -- (1,6);
	\draw (1,6) -- (3,7);
	\draw (3,7) -- (3,8);
	\draw (5,6) -- (3,7);
	\draw (5,6) -- (5,5);
	\draw (5,5) -- (4,4);
	\draw (5,5) -- (6,4);
	\draw (4,4) -- (3,3);
	\draw (4,4) -- (5,3);
	\draw (5,3) -- (5,2);
	\draw (4,1) -- (5,2);
	\draw (6,1) -- (5,2);
	
	\draw[->] (3,8) to[bend right=40] (0.00, 4.08);
    \draw[->] (3,8) to[bend left =60] (6.04, 1.08);
    \draw[->] (1,6) to[bend left =20] (2.01, 4.08);
    \draw[->] (5,6) to[bend right=40] (3.01, 3.08);
    \draw[->] (5,6) to[bend left =20] (6.01, 4.08);
    \draw[->] (5,3) to[bend right=20] (4.01, 1.08);

	\coordinate[label=0: 2] (z) at (0,3.8);
	\coordinate[label=0: 1] (z) at (2,3.8);
	\coordinate[label=0: 1] (z) at (6,3.8);
	\coordinate[label=0: 1] (z) at (3,2.8);
	\coordinate[label=0: 1] (z) at (4,0.8);
	\coordinate[label=0: 3] (z) at (6,0.8);
	\end{tikzpicture}
	\end{minipage}%
	 \hfill%
  \begin{minipage}[t]{.55\linewidth}
	\centering
	\begin{tikzpicture}[scale=0.6, >=latex] 
	\fill (3,8) circle (0.1);
	\fill (3,7) circle (0.1);
	\fill (0,4) circle (0.1);
	\fill (1,5) circle (0.1);
	\fill (1,6) circle (0.1);
	\fill (2,4) circle (0.1);
	\fill (5,6) circle (0.1);
	\fill (5,5) circle (0.1);
	\fill (4,4) circle (0.1);
	\fill (6,4) circle (0.1);
	\fill (3,3) circle (0.1);
	\fill (5,3) circle (0.1);
	\fill (5,2) circle (0.1);
	\fill (4,1) circle (0.1);
	\fill (6,1) circle (0.1);
	
	\draw (0,4) -- (1,5);
	\draw (1,5) -- (2,4);
	\draw (1,5) -- (1,6);
	\draw (1,6) -- (3,7);
	\draw (3,7) -- (3,8);
	\draw (5,6) -- (3,7);
	\draw (5,6) -- (5,5);
	\draw (5,5) -- (4,4);
	\draw (5,5) -- (6,4);
	\draw (4,4) -- (3,3);
	\draw (4,4) -- (5,3);
	\draw (5,3) -- (5,2);
	\draw (4,1) -- (5,2);
	\draw (6,1) -- (5,2);
	
	\draw[->] (3,8) to[bend right=40] (0.00, 4.08);
    \draw[->] (3,8) to[bend left =60] (6.04, 1.08);
    \draw[->] (1,6) to[bend left =20] (2.01, 4.08);
    \draw[->] (5,6) to[bend right=40] (3.01, 3.08);
    \draw[->] (5,6) to[bend left =20] (6.01, 4.08);
    \draw[->] (5,3) to[bend right=20] (4.01, 1.08);

	\coordinate[label=0: 2] (z) at (0,3.8);
	\coordinate[label=0: 2] (z) at (2,3.8);
	\coordinate[label=0: 2] (z) at (6,3.8);
	\coordinate[label=0: 2] (z) at (3,2.8);
	\coordinate[label=0: 3] (z) at (4,0.8);
	\coordinate[label=0: 3] (z) at (6,0.8);
	\end{tikzpicture}
	\end{minipage}	}
\caption{The lambda-DAG representing the term $\lambda x.((\lambda y. xy)(\lambda z. (z(\lambda t. tx))z))$, where the leaves are labelled with (left) the corresponding De Bruijn indices, and (right) the De Bruijn level in which they are located.}
\label{fig:lambdatrees2}
\end{figure}
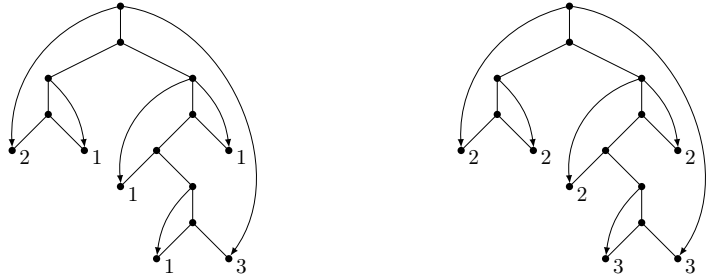

Since the skeleton of a lambda-DAG is a tree, we sometimes call the variables leaves
(\textit{i.e.}, the nodes with out-degree zero), and the path connecting the root with a leaf
(consisting of undirected edges) is called a branch. There are different approaches as to how one
can define the size of a lambda-term (\cite{DGKRT13,bodini2015number, MR3018087}), but within this paper the size will be defined as the number of nodes in the corresponding lambda-DAG.

As mentioned at the beginning, recently, rising interest in the number and structural properties
of lambda-terms can be observed, due to the direct relationship between these random structures
acting as computer programs and mathematical proofs (\cite{curry1958combinatory}).
At first sight lambda-terms appear to be very simple structures, in the sense that their construction can easily be described, but so far no one has yet accomplished to derive their asymptotic number. However, the asymptotic equivalent of the logarithm of this number can be determined up to the second-order term (see \cite{MR3158269}). The difficulty of counting unrestricted lambda-terms arises due to the fact that their number increases superexponentially with increasing size. Thus, if we translate the counting problem into generating functions, then the resulting generating function has a radius of convergence equal to zero, which makes the common methods of analytic combinatorics inapplicable.
This fast growth of the number of lambda-terms can be explained by the numerous possible bindings of leaves by lambdas,  \textit{i.e.} by unary nodes. Consequently, lately some simpler subclasses of lambda-terms, which reduce these multiple binding possibilities, have been studied, e.g. lambda-terms with prescribed number of unary nodes (\cite{bodini2015number}), or lambda-terms in which every lambda binds a prescribed (\cite{MR3158269,MR3101704,MR3063045}) or a bounded (\cite{MR3248351,MR3101704,MR3063045}) number of leaves. 
In this paper we will investigate structural properties of lambda-terms that have been introduced in \cite{MR2815481} and \cite{bodini2015number}, namely at first lambda-terms with a bounded number of abstractions between each leaf and its binding lambda, which corresponds to a bounded De Bruijn index.
The second class of lambda-terms that we will investigate within this paper is the class of lambda-terms with a bounded number of nesting levels of abstractions, \textit{i.e.}, lambda-terms with a bounded number of De Bruijn levels. 
From a practical point of view these restrictions appear to be very natural, since the number of abstractions in lambda-terms which are used for computer programming  is in general assumed to be very low compared to their size (\cite{yang2011finding}).

Particular interest lies in the number and distribution of the variables within these special subclasses of lambda-terms. We will show within this paper that the total number of leaves (\textit{i.e.}, variables) in lambda-DAGs with bounded De Bruijn indices as well as in lambda-terms with bounded number of De Bruijn levels is asymptotically normally distributed. For the latter class of lambda-terms we will also investigate the number of leaves in the different De Bruijn levels, which shows a very interesting behaviour. We will see that in the lower De Bruijn levels, \textit{i.e.} near the root of the lambda-DAG, there are very few leaves, while almost all of the leaves are located in the upper De Bruijn levels and these two domains will turn out to be asymptotically strictly separated. The same behaviour can be shown for unary and binary nodes, which allows us to set up a very interesting ``unary profile'' of this class of lambda-terms.

For lambda-terms that are locally restricted by a bound for the De Bruijn indices the number of De Bruijn levels is not bounded and will tend to infinity for increasing size. The expected number of De Bruijn levels is unknown, which implies that the correct scaling cannot be determined. Thus, we have not been able to establish results concerning the leaves (or other types of nodes) on the different De Bruijn levels for this class of lambda-terms so far. Nevertheless, further studies on this subject seem to be very interesting.

\bigskip
The plan of the paper is as follows: We will present the main results that have been derived in this paper, including all the definitions that are necessary for their understanding, in Section \ref{ch:mainresults}, while the subsequent sections are concerned with their proofs. In Section \ref{ch:length} we will show that the total number of variables in lambda-terms with bounded De Bruijn index is asymptotically normally distributed with mean and variance asymptotically $Cn$ and $\tilde{C}n$, respectively, where the constants $C$ and $\tilde{C}$ depend on the bound that has been imposed. Section \ref{ch:totalheight} shows the same result for lambda-terms where the number of De Bruijn levels is bounded. Finally, in the last section, Section \ref{ch:levelheight}, we show how the variables are distributed in lambda-terms with bounded number of De Bruijn levels. We will see that there are very few leaves on the lower De Bruijn levels, \textit{i.e.}, close to the root, while on the upper De Bruijn levels farther away from the root, there are many leaves. Furthermore, these two domains are strictly separated and we know exactly which is the first level containing a large number of leaves, since this level can be determined by the imposed bound of the number of De Bruijn levels.
This interesting behaviour also holds for the number of binary and unary nodes. By investigating all these numbers among the different De Bruijn levels we are able to set up a so-called unary profile that shows that these special lambda-terms have a very specific shape. A random closed lambda-term with a bounded number of De Bruijn levels starts with a string of unary nodes, where the length of this string depends on the imposed bound. Then it gets slowly filled with nodes until it reaches the aforementioned separating level, where it suddenly starts to contain a lot of nodes.

\section{Main results} 
\label{ch:mainresults}

In this section we will introduce the basic definitions and summarize the main results that will be presented in this paper.

First, we will investigate the total number of variables in lambda-terms with bounded De Bruijn index, \textit{i.e.}, with a bounded number of abstractions between each leaf and its binding lambda. 
Our first main result concerns the asymptotic distribution of the number of variables within this class of closed lambda-terms.

\begin{theo}
\label{theo:mainresultlength}
Let $X_n$ be the total number of variables in a random closed lambda-term of size $n$ where the De Bruijn index of each variable is at most $k$.
Then $X_n$ is asymptotically normally distributed with
\[ \mathbb{E}{X_n} \sim \frac{k}{\sqrt{k}+2k}n, \ \ \ \
\text{and} \ \ \ \ \mathbb{V}X_n \sim \frac{k^2}{2 \sqrt{k}(\sqrt{k}+2k)^2}n, \ \ \ \ \text{as} \ n \rightarrow \infty.\]
\end{theo}

\begin{rem}
Note that $\mathbb{E}{X_n} \tend \frac{n}{2}$ and $\mathbb{V}X_n \tend 0$ for $k \rightarrow \infty$. Since these values are known for the number of leaves in binary trees, this gives a hint that almost all leaves of a large random unrestricted lambda-term are located within an almost purely binary structure. 
\end{rem}

Next we turn to lambda-terms with a bounded number of De Bruijn levels, \textit{i.e.} with a bounded number of unary nodes (or abstractions, respectively) in the separate branches of the corresponding lambda-DAG.

\begin{theo}
\label{theo:totalnumberleavesheight}
Let $\rho_k(u)$ be the root of smallest modulus of the function $z \mapsto R_{j+1,k}(z,u)$, where
\[R_{j+1,k}(z,u)= 1-4(k-j)z^2u-2z+2z\sqrt{1-4(k-j+1)z^2u-2z+ \sqrt{... +2z\sqrt{1-4kz^2u}}},\] and let us define $B(u)=\frac{\rho_k(1)}{\rho_k(u)}$. 

If $B''(1)+B'(1)-B'(1)^2 \neq 0$, then the total number of leaves in closed lambda-DAGs with at most $k$ De Bruijn levels is asymptotically normally distributed with asymptotic mean $\mu n$ and asymptotic variance $\sigma^2 n$, where
$\mu = B'(1)$ and $\sigma^2=B''(1)+B'(1)-B'(1)^2$.
\end{theo}

\begin{rem}
The requirement $B''(1)+B'(1)-B'(1)^2 \neq 0$ obviously results from the fact that otherwise the variance would be $o(n)$. However, this inequality seems to be very difficult to verify, since $B(u)=\frac{\rho_k(1)}{\rho_k(u)}$ and we do not know anything about the function $\rho_k(u)$, except for some crude bounds and its analyticity. But numerical data supports the conjecture that $B''(1)+B'(1)-B'(1)^2 \neq 0$ always holds (\textit{cf}. Table \ref{tab:initialvalues}).
\end{rem}

Lambda-terms with bounded number of De Bruijn levels have been studied in \cite{bodini2015number}, where a very unusual behaviour has been discovered. The asymptotic behaviour of the number of lambda-terms belonging to this subclass differs depending on whether the imposed bound is an element of a certain sequence $(N_i)_{i \geq 0}$, which will be given in Definition \ref{def:sequences}, or not. Though the behaviour of the counting sequence differs for these two cases, the result in Theorem \ref{theo:totalnumberleavesheight} concerning lambda-terms with bounded number of De Bruijn levels is the same after all. However, the method of proof is different in the two cases.
For our subsequent results the distinction of cases will have an impact on the asymptotic behaviour of the counting sequence of the investigated structures. Thus,
we will have to distinguish between these two cases.

\begin{df}[{auxiliary sequences $(u_i)_{i \geq 0}$ and $(N_i)_{i \geq 0}$, \cite[Def.6]{bodini2015number}}]
\label{def:sequences}
Let $(u_i)_{i \geq 0}$ be the integer sequence defined by
\[ u_0=0, \ \ \ u_{i+1}=u_i^2+i+1 \ \ \ \text{for} \ i \geq 0, \]
and $(N_i)_{i \geq 0}$ by
\[ N_i = u_i^2-u_i+i, \ \ \ \text{for all} \ i \geq 0. \]
\end{df}

In the last section we investigate the distribution of the different types of nodes in lambda-DAGs with bounded number of De Bruijn levels among the separate levels throughout the DAG.

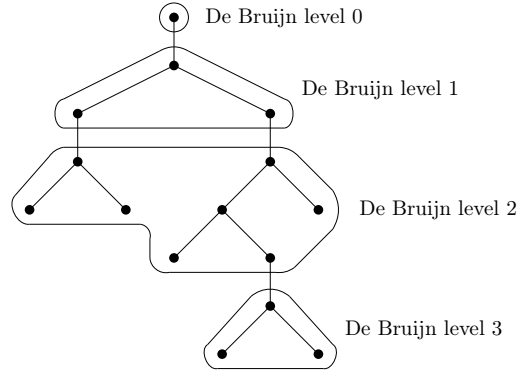
\begin{figure}[h!]
  \centering
\scalebox{0.8}{	   \begin{tikzpicture}[scale=0.8]
	\fill (3,8) circle (0.1);
	\fill (3,7) circle (0.1);
	\fill (0,4) circle (0.1);
	\fill (1,5) circle (0.1);
	\fill (1,6) circle (0.1);
	\fill (2,4) circle (0.1);
	\fill (5,6) circle (0.1);
	\fill (5,5) circle (0.1);
	\fill (4,4) circle (0.1);
	\fill (6,4) circle (0.1);
	\fill (3,3) circle (0.1);
	\fill (5,3) circle (0.1);
	\fill (5,2) circle (0.1);
	\fill (4,1) circle (0.1);
	\fill (6,1) circle (0.1);
	
	\draw (0,4) -- (1,5);
	\draw (1,5) -- (2,4);
	\draw (1,5) -- (1,6);
	\draw (1,6) -- (3,7);
	\draw (3,7) -- (3,8);
	\draw (5,6) -- (3,7);
	\draw (5,6) -- (5,5);
	\draw (5,5) -- (4,4);
	\draw (5,5) -- (6,4);
	\draw (4,4) -- (3,3);
	\draw (4,4) -- (5,3);
	\draw (5,3) -- (5,2);
	\draw (4,1) -- (5,2);
	\draw (6,1) -- (5,2);
	\coordinate[label=0: De Bruijn level 0] (z) at (3.5,8);
	\draw (3,8) circle (0.3);
	\coordinate[label=0: De Bruijn level 1] (z) at (5.5,6.5);
	\draw (0.7,5.7) -- (5.3,5.7);
	\draw (3.3,7.3) -- (5.3,6.3);
	\draw (0.7,6.3) -- (2.7,7.3);
	\draw (0.7,5.7) to[bend left=80] (0.7,6.3);
	\draw (5.3,5.7) to[bend right=80] (5.3,6.3);
	\draw (3.3,7.3) to[bend right=30] (2.7,7.3);
	\coordinate[label=0: De Bruijn level 2] (z) at (6.7,4);
	\draw (0.7,5.3) -- (5.2,5.3);
	\draw (0.4,5.1) -- (-0.3,4.3);
	\draw (0,3.7) -- (2.3,3.7);
	\draw (2.5,3.4) -- (2.5,3);
	\draw (2.9,2.7) -- (5.1,2.7);
	\draw (5.5,2.8) -- (6.3,3.6);
	\draw (6.3,4.4) -- (5.6,5.1);
	\draw (0.7,5.3) to[bend right=20] (0.4,5.1);
	\draw (-0.3,4.3) to[bend right=70] (0,3.7);
	\draw (2.3,3.7) to[bend left=50] (2.5,3.4);
	\draw (2.9,2.7) to[bend left=50] (2.5,3);
	\draw (5.1,2.7) to[bend right=20] (5.5,2.8);
	\draw (6.3,3.6) to[bend right=30] (6.3,4.4);
	\draw (5.6,5.1) to[bend right=20] (5.2,5.3);
	\coordinate[label=0: De Bruijn level 3] (z) at (6.4,1.5);
	\draw (3.9,0.7) -- (6.1,0.7);
	\draw (3.7,1.2) -- (4.6,2.2);
	\draw (5.4,2.2) -- (6.3,1.2);
	\draw (3.9,0.7) to[bend left=80] (3.7,1.2);
	\draw (6.1,0.7) to[bend right=80] (6.3,1.2);
	\draw (5.4,2.2) to[bend right=40] (4.6,2.2);
	\end{tikzpicture} }
	
\caption{Underlying Motzkin tree of e.g. the lambda term $\lambda x. ((\lambda y.yx)(\lambda z.(z(\lambda t. tx))z))$, where the different De Bruijn levels are encircled.}
\label{fig:unarylevels}
\end{figure}

\begin{rem}
Note that the De Bruijn level in which a node is located just counts the number of unary nodes in the branch connecting the root and the respective node.
\end{rem}

The following theorem includes the results that we will present in Section \ref{sec:leaveslevels}, where we show that the number of leaves near the root of the lambda-DAG, \textit{i.e.}, in the lower De Bruijn levels, is very low, while there are many leaves in the upper levels. Furthermore these two domains are strictly separated and the ``separating level'', \textit{i.e.}, the first level with many leaves, depends on the bound of the number of De Bruijn levels. We will show a very interesting behaviour, namely that with growing bound the number of leaves within the De Bruijn level that is directly below the critical separating level increases, until the bound reaches a certain number, which makes this adjacent leaf-filled level become the new separating level. Thus, we can observe a ``double jump'' in the asymptotic behaviour of the number of leaves within the separate levels (\textit{cf}. Figure \ref{fig:levelsdoublejump}).

\begin{theo}
\label{theo:levelsmeandist}
Let $_{k-l}\tilde{H}_k(z,u)$ denote the bivariate generating function of the class of closed lambda-terms with at most $k$ De Bruijn levels, where $z$ marks the size and $u$ marks the number of leaves in the $(k-l)$-th De Bruijn level. Additionally, we denote its dominant singularity by $\tilde{\rho}_k(u)$, and $\tilde{B}(u)=\frac{\tilde{\rho}_k(u)}{\tilde{\rho}_k(1)}$. Then the following assertions hold:
\noindent
\begin{itemize}[leftmargin=*]
\item If $k \in (N_j,N_{j+1})$, then the average number of leaves in the first $k-j$ De Bruijn levels is $O(1)$, as the size $n \rightarrow \infty$, while it is $\Theta(n)$ for the last $j+1$ levels.

In particular, if $\tilde{B}''(1)+\tilde{B}'(1)-\tilde{B}'(1)^2 \neq 0$, the number of leaves in each of the last $j+1$
De Bruijn levels is asymptotically normally distributed with mean and variance proportional to the size $n$ of the lambda-term.

\item If $k = N_j$, then the average number of leaves in the first $k-j$ De Bruijn levels is $O(1)$, as $n \rightarrow \infty$, while the average number of leaves in the $(k-j)$-th level is $\Theta(\sqrt{n})$. The last $j$ De Bruijn levels have asymptotically $\Theta(n)$ leaves.

In particular, if $\tilde{B}''(1)+\tilde{B}'(1)-\tilde{B}'(1)^2 \neq 0$, the number of leaves in each of the last $j$ De Bruijn levels is asymptotically normally distributed with mean and variance proportional to the size $n$ of the lambda-term..
\end{itemize}
\end{theo}

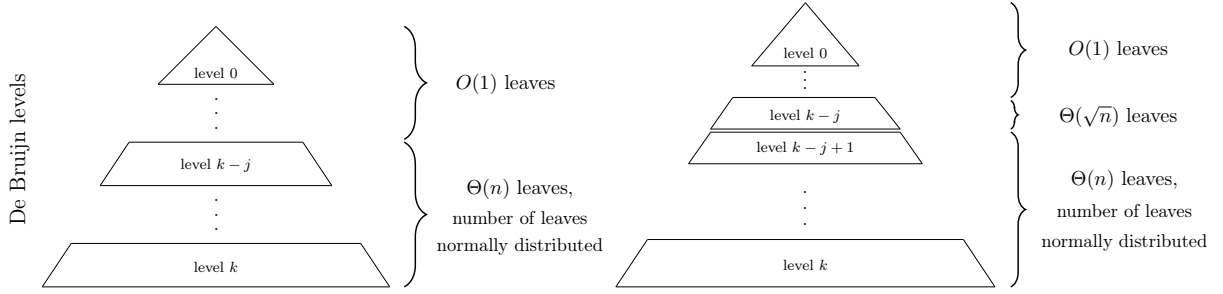
\begin{figure}[h!]
  \centering
\scalebox{0.7}{\begin{tikzpicture}[scale=0.55] 
	\draw (6,9) -- (4,7);
	\draw (6,9) -- (8,7);
	\draw (4,7) -- (8,7);
	\coordinate[label=0: .] (z) at (5.7,6.5);
	\coordinate[label=0: .] (z) at (5.7,6);
	\coordinate[label=0: .] (z) at (5.7,5.5);
	\draw (3,5) -- (9,5);
	\draw (9,5) -- (10,3.5);
	\draw (2,3.5) -- (10,3.5);
	\draw (3,5) -- (2,3.5);
	\coordinate[label=0: .] (z) at (5.7,3);
	\coordinate[label=0: .] (z) at (5.7,2.5);
	\coordinate[label=0: .] (z) at (5.7,2);
	\draw (0,0) -- (12,0);
	\draw (0,0) -- (1,1.5);
	\draw (1,1.5) -- (11,1.5);
	\draw (11,1.5) -- (12,0);
	\node at (6,7.4) {\footnotesize{level $0$}};
	\node at (6,4.2) {\footnotesize{level $k-j$}};
	\node at (6,0.7) {\footnotesize{level $k$}};
	\node at (16,7) {\large{$O(1)$ leaves}};
	\node at (16.5,3.4) {\large{$\Theta(n)$ leaves},};
	\node at (16.5,2.4) {{number of leaves}};
	\node at (16.5,1.4) {{normally distributed}};
	\node[label={[label distance=0.5cm,rotate=90]right:\Large{De Bruijn levels}}] at (-1,1) {};	\draw[thick,black,decorate,decoration={brace,amplitude=12pt}] (12.5,9) -- (12.5,5.1) node[midway, above,yshift=12pt,]{};	\draw[thick,black,decorate,decoration={brace,amplitude=12pt}] (12.5,5) -- (12.5,0) node[midway, above,yshift=12pt,]{};
	\end{tikzpicture}
\begin{tikzpicture}[scale=0.6] 
	\draw (6,9) -- (4.3,7);
	\draw (6,9) -- (7.7,7);
	\draw (4.3,7) -- (7.7,7);
	\coordinate[label=0: .] (z) at (5.7,6.8);
	\coordinate[label=0: .] (z) at (5.7,6.55);
	\coordinate[label=0: .] (z) at (5.7,6.3);
	\draw (3.7,6) -- (8.2,6);
	\draw (8.2,6) -- (9,5);
	\draw (3,5) -- (9,5);
	\draw (3.7,6) -- (3,5);
	
	\draw (2.3,3.9) -- (9.7,3.9);
	\draw (9.7,3.9) -- (9,4.9);
	\draw (3,4.9) -- (9,4.9);
	\draw (2.3,3.9) -- (3,4.9);
	\coordinate[label=0: .] (z) at (5.7,3);
	\coordinate[label=0: .] (z) at (5.7,2.5);
	\coordinate[label=0: .] (z) at (5.7,2);
	\draw (0,0) -- (12,0);
	\draw (0,0) -- (1,1.5);
	\draw (1,1.5) -- (11,1.5);
	\draw (11,1.5) -- (12,0);
	\node at (6,7.4) {\footnotesize{level $0$}};
	\node at (6,5.4) {\footnotesize{level $k-j$}};
	\node at (6,4.4) {\footnotesize{level $k-j+1$}};
	\node at (6,0.7) {\footnotesize{level $k$}};
	\node at (15.9,7.5) {\large{$O(1)$ leaves}};
	\node at (15.9,5.4) {\large{$\Theta(\sqrt{n})$ leaves}};
	\node at (16.1,3.4) {\large{$\Theta(n)$ leaves},};	
	\node at (16.1,2.4) {{number of leaves}};
	\node at (16.1,1.4) {{normally distributed}};\draw[thick,black,decorate,decoration={brace,amplitude=4pt}] (12.5,5.9) -- (12.5,5) node[midway, above,yshift=4pt,]{};
\draw[thick,black,decorate,decoration={brace,amplitude=8pt}] (12.5,9) -- (12.5,6) node[midway, above,yshift=8pt,]{};	\draw[thick,black,decorate,decoration={brace,amplitude=8pt}] (12.5,4.9) -- (12.5,0) node[midway, above,yshift=8pt,]{};
	\end{tikzpicture}}
\caption{Summary of the mean values of the number of leaves in the different De Bruijn levels in lambda-terms with at most $k$ De Bruijn levels for the case $N_j < k < N_{j+1}$ (left), and the case $k=N_j$ (right).}
\label{fig:levelsdoublejump}
\end{figure}

%

Sections \ref{sec:unarylevels} and \ref{sec:binary} are concerned with the investigation of the number of unary nodes, and binary nodes respectively, among the De Bruijn levels. Using the same techniques as in Section \ref{sec:leaveslevels} we can show that their number behaves in fact very similar to the number of leaves.

\begin{theo}
\label{theo:levelsmeanbinary}
If $k \in (N_j,N_{j+1})$, then both the average number of unary nodes and the avergae number of binary nodes in the first $k-j$ De Bruijn levels are $O(1)$, as $n \rightarrow \infty$, while they are $\Theta(n)$ in each of the last $j+1$ levels.

If $k = N_j$, then both the average number of unary nodesand the average number of binary nodes in the first $k-j$ De Bruijn levels is $O(1)$, as $n \rightarrow \infty$, while the average number of nodes of the respective type in the $(k-j)$-th De Bruijn level is $\Theta(\sqrt{n})$. The last $j$ De Bruijn levels contain each asymptotically $\Theta(n)$ unary nodes, as well as $\Theta(n)$ binary nodes.
\end{theo}



\section{Total number of leaves in lambda-terms with bounded De Bruijn indices} 
\label{ch:length}

In this section we investigate the asymptotic number of all leaves in closed lambda-terms with bounded De Bruijn indices.
In order to get some quantitative results concerning this restricted class of lambda-terms we will use the well-known symbolic method (see \cite{MR2483235}) and therefore we introduce further combinatorial classes as it has been done in \cite{bodini2015number}: $\mathcal{Z}$ denotes the class of atoms, $\mathcal{A}$ the class of application nodes ({\textit{i.e.}, binary nodes), $\mathcal{U}$ the class of abstraction nodes (\textit{i.e.}, unary nodes), and $\hat{\mathcal{P}}^{(i,k)}$ the class of unary-binary trees such that every leaf $e$ can be labelled in $\min \{ h_u(e)+i,k \}$ ways. 

The classes $\hat{\mathcal{P}}^{(i,k)}$ can be specified by

\[ \hat{\mathcal{P}}^{(k,k)} = k \mathcal{Z} + ( \mathcal{A} \times \hat{\mathcal{P}}^{(k,k)} \times \hat{\mathcal{P}}^{(k,k)} ) + ( \mathcal{U} \times \hat{\mathcal{P}}^{(k,k)}), \ \ \ \ \ \ \ \ \ \ \ \ \ \ \  \ \ \ \ \ \]
and
\[ \hat{\mathcal{P}}^{(i,k)} = i \mathcal{Z} + ( \mathcal{A} \times \hat{\mathcal{P}}^{(i,k)} \times \hat{\mathcal{P}}^{(i,k)} ) + ( \mathcal{U} \times \hat{\mathcal{P}}^{(i+1,k)}) \ \ \ \ \ \ \ \text{for} \ \ i < k. \]

Translating into generating functions with $z$ marking the size and $u$ marking the number of leaves, we get
\[ \hat{P}^{(k,k)}(z,u) = kzu + z \hat{P}^{(k,k)^2}(z,u) + z\hat{P}^{(k,k)}(z,u), \]
and 
\[ \hat{P}^{(i,k)}(z,u) = izu + z \hat{P}^{(i,k)^2}(z,u) + z\hat{P}^{(i+1,k)}(z,u), \]
which yields
\[ \hat{P}^{(k,k)}(z,u) = \frac{1-z- \sqrt{(1-z)^2-4z^2ku}}{2z}, \]
and
\[ \hat{P}^{(i,k)}(z,u) = \frac{1-\sqrt{1-4z(izu+z\hat{P}^{(i+1,k)}(z,u))}}{2z} = \frac{1-\sqrt{1-4iz^2u-4z^2\hat{P}^{(i+1,k)}(z,u)}}{2z}, \]
for $i<k$.

This can be written in the form

\[ \hat{P}^{(i,k)}(z,u) = \frac{1- \mathbf{1}_{[i=k]}z-\sqrt{\hat{R}_{k-i+1,k}(z,u)}}{2z}, \]
with
\begin{align}
\label{equ:radicalhatk1}
 \hat{R}_{1,k}(z,u) = (1-z)^2-4kuz^2, 
\end{align}
\[ \hat{R}_{2,k}(z,u) = 1-4(k-1)z^2u-2z+2z^2+2z \sqrt{\hat{R}_{1,k}(z,u)}, \]
and
\begin{align}
\label{equ:radicalhatrecursion}
\hat{R}_{i,k}(z,u) = 1-4(k-i+1)z^2u-2z+2z \sqrt{\hat{R}_{i-1,k}(z,u)}, \ \ \ \ \  \text{for} \ \ 3 \leq i \leq k+1. 
\end{align}

Since the class $\hat{\mathcal{P}}^{(0,k)}$ is isomorphic to the class $\mathcal{G}_{k}$ of closed lambda-terms where all De Bruijn indices are not larger than $k$, we get for the corresponding bivariate generating function
\[ G_k(z,u)= \hat{P}^{(0,k)}(z,u)= \frac{1- \sqrt{\hat{R}_{k+1,k}(z,u)}}{2z}. \]

From \cite{bodini2015number} we know that the dominant singularity  of $G_{k}(z,1)$ comes from the innermost radicand only and consequently is of type $\frac{1}{2}$. Due to continuity arguments this implies that in a sufficiently small neighbourhood of $u=1$ the dominant singularity $\hat{\rho}_k(u)$ of $G_{k}(z,u)$ comes also only from the innermost radicand, \textit{i.e.}, $\hat{R}_{1,k}(z,u)$, and is of type $\frac{1}{2}$. By calculating the smallest positive root of $\hat{R}_{1,k}(z,u)$ we get $\hat{\rho}_k(u)= \frac{1}{1+2\sqrt{ku}}$. Now we will determine the expansions of the radicands in a neighbourhood of the dominant singularity $\hat{\rho}_k(u)$.

\begin{prop}
\label{prop:rjkaroundsing}
Let $\hat{\rho}_k(u)$ be the root of the innermost radicand $\hat{R}_{1,k}(z,u)$, \textit{\textit{i.e.}} $\hat{\rho}_k(u) = \frac{1}{1+2\sqrt{ku}}$, where $u$ is in a sufficiently small neighbourhood of 1, \textit{i.e.} $|u-1|<\delta$ for $\delta>0$ sufficiently small. Then the equations 
\begin{align}
\label{equ:r1karoundsing}
 \hat{R}_{1,k}(\hat{\rho}_k(u)-\epsilon,u) = \big (2 -2\hat{\rho}_k(u) +8ku\hat{\rho}_k(u) \big) \epsilon + \mathcal{O}(|\epsilon|^2), 
\end{align}
and
\begin{align} 
\label{equ:rjkaroundsing}
\hat{R}_{j,k}(\hat{\rho}_k(u)-\epsilon,u) = {c}_j(u) \hat{\rho}_k(u)^2 + \frac{\sqrt{8\hat{\rho}_k(u) (1-\hat{\rho}_k(u)^2)}}{\prod_{l=2}^{j}\sqrt{c_l(u)}} \cdot \sqrt{\epsilon} + \mathcal{O}(|\epsilon|),
\end{align}
for $2 \leq j \leq k+1$, with ${c}_1(u)=1$ and ${c}_j(u)=4(j-1)u-1+2\sqrt{c_{j-1}(u)}$ for $2 \leq j \leq k+1$, hold for $\epsilon \longrightarrow 0$ so that $\epsilon \in \mathbb{C} \setminus \mathbb{R}^{-}$, uniformly in $u$.
\end{prop}

\begin{proof}
Using the Taylor expansion of $\hat{R}_{1,k}(z,u)$ around $\hat{\rho}_k(u)$ we obtain
\[ \hat{R}_{1,k}(z,u) = \hat{R}_{1,k}(\hat{\rho}_k(u),u)+(z-\hat{\rho}_k(u)) \frac{\partial}{\partial z}\hat{R}_{1,k}(\hat{\rho}_k(u),u)+ \mathcal{O}((z-\hat{\rho}_k(u))^2). \]

Per definition the first summand $\hat{R}_{1,k}(\hat{\rho}_k(u),u)$ is equal to zero. Setting $z=\hat{\rho}_k(u)-\epsilon$ and using (\ref{equ:radicalhatk1}) we obtain the first claim of Proposition \ref{prop:rjkaroundsing}.

The next step is to compute an expansion of $\hat{R}_{j,k}(z,u)$ around $\hat{\rho}_k(u)$ for $2\leq j \leq k+1$. 
  Using the recursive relation (\ref{equ:radicalhatk1}) for $\hat{R}_{2,k}(z,u)$ and the formula $\hat{\rho}_k(u)=\frac{1}{1+2\sqrt{ku}}$ yields
\[ 
\hat{R}_{2,k}(\hat{\rho}_k(u)-\epsilon,u) = (1+4u) \hat{\rho}_k(u)^2 + 2\hat{\rho}_k(u)\sqrt{2 -2\hat{\rho}_k(u) +8ku\hat{\rho}_k(u)} \sqrt{\epsilon} + \mathcal{O}(|\epsilon|). 
\]

We set $c_2(u):=1+4u$ and $d_2(u):=2\hat{\rho}_k(u)\sqrt{2 -2\hat{\rho}_k(u) +8ku\hat{\rho}_k(u)}$ and assume that for $2 \leq j \leq k+1$ the equation $\hat{R}_{j,k}(\hat{\rho}_k(u)-\epsilon,u)=c_j(u) \hat{\rho}_k^2(u)+d_j(u) \sqrt{\epsilon} + \mathcal{O}(|\epsilon|)$ holds. Now we proceed by induction. 

Observe that
\[ 
\hat{R}_{j+1,k}(\hat{\rho}_k(u)-\epsilon,u) = 1- 4(k-j)\hat{\rho}_k^2(u)u-2\hat{\rho}_k(u)+
2\hat{\rho}_k(u)\sqrt{c_j(u) \hat{\rho}_k^2+d_j(u) \sqrt{\epsilon} + \mathcal{O}(|\epsilon|)}.
\]

Expanding, using again $\hat{\rho}_k(u)=\frac{1}{1+2\sqrt{ku}}$, and simplifying yields

\[ 
\hat{R}_{j+1,k}(\hat{\rho}_k(u)-\epsilon,u) = 4ju\hat{\rho}_k(u)^2-\hat{\rho}_k(u)^2 +2\hat{\rho}_k(u)^2\sqrt{c_j(u)}+ \frac{d_j(u)}{\sqrt{c_j(u)}}\sqrt{\epsilon} + \mathcal{O}(|\epsilon|). 
\]

Setting $c_{j+1}(u):=4ju-1+2\sqrt{c_j(u)}$ and $d_{j+1}(u):=\frac{d_j(u)}{\sqrt{c_j(u)}}$ for $2 \leq j \leq k$, we obtain 
$\hat{R}_{j+1,k}(\hat{\rho}_k(u)-\epsilon,u)=c_{j+1} \hat{\rho}_k^2(u) +d_{j+1} \sqrt{\epsilon} + \mathcal{O}(|\epsilon|)$.

Expanding $d_{j+1}(u)$, using its recursive relation and $d_2(u)=2\hat{\rho}_k(u)\sqrt{2 -2\hat{\rho}_k(u) +8ku\hat{\rho}_k(u)}$, we get for $2 \leq j \leq k$
\[ d_{j+1}(u)= \frac{2\hat{\rho}_k(u)\sqrt{2 -2\hat{\rho}_k(u) +8ku\hat{\rho}_k(u)}}{\prod_{l=2}^j \sqrt{c_l(u)}}. \]

Finally, we show that the $c_l(u)$'s are greater than zero in a neighbourhood of $u=1$. By induction it can easily be seen that they are always positive for $u=1$, since
\begin{align*}
c_1(1)=1,
\end{align*}
and assuming $c_{i-1}(1)<c_{i}(1)$ we get
\begin{align*}
c_{i+1}(1)=41-1+2\sqrt{c_i(1)} > 4(i-1) + 4 - 1 + 2\sqrt{c_{i-1}(1)} = c_i(1) + 4.
\end{align*}
Using continuity arguments we can see that the functions $c_l(u)$ have to be positive in a sufficiently small neighbourhood of $u=1$ as well, which completes the proof of (\ref{equ:rjkaroundsing}).
\end{proof}

\begin{theo}
\label{thm:asympgleqzu}
Let for any fixed $k$, $G_{k}(z,u)$ denote the bivariate generating function of the class of closed lambda-terms where all De Bruijn indices are at most $k$. Then the equation 
\[ [z^n]G_{k}(z,u) = \sqrt{\frac{\sqrt{ku}+2ku}{4\pi\prod_{l=2}^{k+1}c_l(u)}} (1+2\sqrt{ku})^n n^{-\frac{3}{2}} \left( 1+ O \left( \frac{1}{\sqrt{n}} \right) \right), \ \ \ \ \text{for} \ \ n \rightarrow \infty, \]
with $c_1(u)=1$ and $c_j(u)=4(j-1)u-1+2\sqrt{c_{j-1}(u)}$, for $2 \leq j \leq k+1$, holds uniformly in $u$ for $|u-1|<\delta$, with $\delta>0$ sufficiently small.
\end{theo}

\begin{proof} 
Using $G_{k}(z,u)= \frac{1- \sqrt{\hat{R}_{k+1,k}(z,u)}}{2z}$ and (\ref{equ:rjkaroundsing}), we get for $\epsilon \in \mathbb{C} \setminus \mathbb{R}^{-}$ with $|\epsilon| \longrightarrow 0$
\[ G_{k}(\hat{\rho}_k(u)-\epsilon,u)= \frac{1-\sqrt{c_{k+1}(u)}\hat{\rho}_k(u)}{2\hat{\rho}_k(u)} - \frac{d_{k+1}(u)}{4\hat{\rho}_k(u)^2\sqrt{c_{k+1}(u)}} \sqrt{\epsilon} + \mathcal{O}(|\epsilon|). \]

  Hence,
\[ [z^n]G_{k}(z,u) = - \frac{d_{k+1}(u)\sqrt{\hat{\rho}_k(u)}}{4\hat{\rho}_k^2(u) \sqrt{c_{k+1}(u)}} [z^n]  \sqrt{1- \frac{z}{\hat{\rho}_k(u)}} + [z^n] \mathcal{O} \left( \hat{\rho}_k(u) - z  \right). \]

  Since $\hat{\rho}_k(u)=\frac{1}{1+2\sqrt{ku}}$ is of type $\frac{1}{2}$ and by plugging in the formula for $d_{k+1}(u)=\frac{2\hat{\rho}_k(u)\sqrt{2 -2\hat{\rho}_k(u) +8ku\hat{\rho}_k(u)}}{\prod_{l=2}^k \sqrt{c_l(u)}}$, we obtain the desired result by applying singularity analysis.
\end{proof}

From \cite[Theorem 1]{bodini2015number} we know the following result:

\begin{align}
\label{resultgleqkasymp}
[z^n]G_{k}(z,1) = \sqrt{\frac{\sqrt{k}+2k}{4\pi\prod_{l=2}^{k+1}c_l(1)}} (1+2\sqrt{k})^n n^{-\frac{3}{2}} \left( 1+ O \left( \frac{1}{\sqrt{n}} \right) \right), \ \ \ \text{as} \ n \rightarrow \infty, 
\end{align}
with $c_l(u)$ defined as in Proposition \ref{prop:rjkaroundsing}.

\medskip
Now we want to apply the well-known Quasi-Power Theorem.

\begin{theo}[Quasi-Power Theorem, \cite{hwang1998convergence}]
Let $X_n$ be a sequence of random variables with the property that \[ \mathbb{E}u^{X_n} = A(u) B(u)^{\lambda_n} \left( 1+ \mathcal{O} \left(\frac{1}{\phi_n} \right) \right) \]
holds uniformly in a complex neighbourhood of $u=1$, where $\lambda_n \rightarrow \infty$ and $\phi_n
\rightarrow \infty$, and $A(u)$ and $B(u)$ are analytic functions in a neighbourhood of $u=1$ with
$A(1)=B(1)=1$. Set $\mu = B'(1)$ and $\sigma^2=B''(1)+B'(1)-B'(1)^2$. If $\sigma^2 \neq 0$, then 
\[ \frac{X_n - \mathbb{E}X_n}{\sqrt{\mathbb{V}X_n}} \rightarrow \mathcal N(0,1), \]
with $\mathbb{E}X_n = \mu \lambda_n + A'(1) + \mathcal{O}(1/\phi_n))$ and $\mathbb{V}X_n =
\sigma^2 \lambda_n +A''(1)+A'(1)-A'(1)^2+\mathcal{O}(1/\phi_n))$.
\end{theo}

Using Theorem \ref{thm:asympgleqzu} and (\ref{resultgleqkasymp}), we get for $n \longrightarrow \infty$
\[ \mathbb{E}u^{X_n} = \frac{[z^n]G_{k}(z,u)}{[z^n]G_{k}(z,1)} = \left( \frac{1+2\sqrt{ku}}{1+2\sqrt{k}} \right)^n  \sqrt{ \frac{\sqrt{ku}+2ku}{2k+\sqrt{k}} \prod_{j=2}^{k+1} \frac{{c}_j(1)}{c_j(u)}} \left( 1+ O \left( \frac{1}{n} \right) \right), \]
where $c_1(u)=1$ and $c_j(u)=4ju-4u-1+2\sqrt{c_{j-1}(u)}$.

Thus, all assumptions for the Quasi-Power Theorem are fulfilled, and we get that the number of leaves in closed lambda-terms with De Bruijn indices at most $k$ is asymptotically normally distributed with

\[ \mathbb{E}{X_n} \sim \frac{k}{\sqrt{k}+2k}n, \ \ \ \text{and} \ \ \
 \mathbb{V}X_n \sim \frac{k^2}{2 \sqrt{k}(\sqrt{k}+2k)^2}n, \ \ \text{as} \ n \rightarrow \infty, \]
and therefore Theorem \ref{theo:mainresultlength} is shown.

\section{Total number of leaves in lambda-terms with bounded number of De Bruijn levels} 
\label{ch:totalheight}

This section is devoted to the enumeration of leaves in closed lambda-terms with a bounded number of De Bruijn levels. As in \cite{bodini2015number} let us denote by $\mathcal{P}^{(i,k)}$ the class of unary-binary trees that contain at most $k-i$ De Bruijn levels and each leaf $e$ can be coloured with one out of $i+l(e)$ colors, where $l(e)$ denotes the De Bruijn level in which the respective leaf is located. These classes can be specified by

\[ \mathcal{P}^{(k,k)} = k \mathcal{Z} + ( \mathcal{A} \times \mathcal{P}^{(k,k)} \times \mathcal{P}^{(k,k)} ), \]
and
\[ \mathcal{P}^{(i,k)} = i \mathcal{Z} + ( \mathcal{A} \times \mathcal{P}^{(i,k)} \times \mathcal{P}^{(i,k)} ) + ( \mathcal{U} \times \mathcal{P}^{(i+1,k)} ) \ \ \ \ \ \text{for} \ \ i < k. \]

By translating into generating functions we get

\[ P^{(k,k)}(z,u)= kzu + zP^{(k,k)^2}(z,u), \]
and
\[ P^{(i,k)}(z,u) = izu + zP^{(i,k)^2}(z,u) + zP^{(i+1,k)}(z,u) \ \ \ \ \ \text{for} \ \ i < k. \]

Solving yields
\[ P^{(k,k)}(z,u)=\frac{1- \sqrt{1-4kz^2u}}{2z}, \]
and
\[ P^{(i,k)}(z,u) = \frac{1- \sqrt{1-4iz^2u - 4z^2P^{(i+1,k)}}}{2z} \ \ \ \ \text{for} \ \ i < k. \]
This can be written as
\[ P^{(i,k)}(z,u) = \frac{1- \sqrt{R_{k-i+1,k}(z,u)}}{2z}, \]
where
\begin{align}
\label{radicals1k}
R_{1,k}(z,u)=1-4kz^2u, 
\end{align}
and
\begin{align}
\label{radicalsrecursion}
R_{i,k}(z,u)=1-4(k-i+1)z^2u-2z+2z \sqrt{R_{i-1,k}(z,u)}, \ \ \ \ \text{for} \ \ 2 \leq i \leq k+1. 
\end{align}

For the bivariate generating function of closed lambda-terms with at most $k$ De Bruijn levels we get
\[ H_{k }(z,u) = P^{(0,k)}(z,u) = \frac{1- \sqrt{R_{k+1,k}(z,u)}}{2z}. \]

Thus, the generating function consists again of $k+1$ nested radicals, but as stated in Section \ref{ch:mainresults}, the counting sequence of this class of lambda-terms shows a very unusual behaviour. The type of the dominant singularity of the generating function changes when the imposed bound equals $N_j$. Thus, the subexponential term in the asymptotics of the counting sequence changes. The following result has been shown in \cite{bodini2015number}:

\begin{theo}[{\cite[Theorem 3]{bodini2015number}}]
\label{theo:asymphleqkueg1}
Let $(u_i)_{i \geq 0}$ and $(N_i)_{i \geq 0}$ be the integer sequences defined in
Definition~\ref{def:sequences} and let $H_{k}(z,1)$ be the generating function of the class of
closed lambda-terms with at most $k$ De Bruijn levels. Then the following asymptotic relations hold
\begin{itemize}
	\item[(i)] If there exists $j \geq 0$ such that $N_j < k < N_{j+1}$, then there exists a constant $h_k$ such that
	\[ [z^n]H_k(z,1) \sim h_k n^{-3/2} \rho_k(1)^{-n}, \ \text{as} \ n \rightarrow \infty. \]
	\item[(ii)] If there exists $j$ such that $k=N_j$, then 
	\[ [z^n]H_k(z,1) \sim h_k n^{-5/4} \rho_k(1)^{-n} = h_k n^{-5/4} (2u_j)^{n}, \ \text{as} \ n \rightarrow \infty. \]
	\end{itemize}
\end{theo}

Thus, in order to investigate structural properties of this class of lambda-terms we perform a distinction of cases whether the bound $k$ is an element of the sequence $(N_i)_{i \geq 0}$ or not.

\subsection{The case $N_j < k <N_{j+1}$}
\label{subsec:total_height_njnj+1}

From \cite{bodini2015number} we know that in this case the dominant singularity of the generating function  $H_{k}(z,1)$ comes from the $(j+1)$-th radicand $R_{j+1,k}$ and is of type $\frac{1}{2}$. As in the previous section we can again use continuity arguments to guarantee that sufficiently close to $u=1$ the dominant singularity $\rho_{k}(u)$ of $H_{k}(z,u)$ comes from the $(j+1)$-th radicand $R_{j+1,k}(z,u)$ and is of type $\frac{1}{2}$. Now we will determine the expansions of the radicands in a neighbourhood of the dominant singularity.

\begin{prop}
\label{prop:exprjkaroundsing}
Let $\rho_{k}(u)$ be the dominant singularity of $H_{k}(z,u)$, where $u$ is in a sufficiently small neighbourhood of 1, \textit{i.e.} $|u-1|<\delta$ for $\delta>0$ sufficiently small. Then the expansions 
\begin{itemize}
 	\item[(i)] $\forall i < j+1 \ \text{(inner radicands)}: R_{i,k}(\rho_{k}(u)-\epsilon,u) = R_{i,k}(\rho_{k}(u),u)+ \mathcal{O}(|\epsilon|)$
 	
 	\item[(ii)] $R_{j+1,k}(\rho_{k}(u)-\epsilon,u) =  \gamma_{j+1}(u) \epsilon +\mathcal{O}(|\epsilon|^2)$, with $\gamma_{j+1}(u)=-\frac{\partial}{\partial z} R_{j+1,k}(\rho_{k}(u),u)$
 	
 	\item[(iii)] $\forall i > j+1 \ \text{(outer radicands)}: R_{i,k}(\rho_{k}(u)-\epsilon,u) =a_i(u)+b_i(u)\sqrt{\epsilon}+\mathcal{O}(|\epsilon|),$ 
 	with $a_{i+1}(u)=1-4(k-i)\rho_{k}(u)^2u-2\rho_{k}(u)+2\rho_{k}(u)\sqrt{a_i(u)}$, and $b_{i+1}(u)=\frac{b_i(u) \rho_{k}(u)}{\sqrt{a_i(u)}}$ for $j+2 \leq i \leq k$, with $a_{j+2}(u)=1-4(k-j-1)\rho_{k}(u)^2u-2\rho_{k}(u)$ and $b_{j+2}(u)=2\rho_{k}(u)\sqrt{\gamma_{j+1}(u)}$,
\end{itemize}
hold for $\epsilon \longrightarrow 0$ so that $\epsilon \in \mathbb{C} \setminus \mathbb{R}^{-}$, uniformly in $u$. 
\end{prop}

\begin{proof}
\begin{itemize}
\item[(i)] The first equation (for $i<j+1$) follows immediately by Taylor expansion around $\rho_{k}(u)$ and setting $z=\rho_{k}(u)-\epsilon$.
\item[(ii)] The equation for $i=j+1$ follows analogously to the first case, knowing that $R_{j+1,k}(z,u)$ cancels for $z=\rho_{k}(u)$.
\item[(iii)] The next step is to expand $R_{i,k}(z,u)$ around $\rho_{k}(u)$ for $i >j+1$. From the second claim of Proposition \ref{prop:exprjkaroundsing} and from the recurrence relation (\ref{radicalsrecursion}) for $R_{i,k}(z,u)$ it results
\[ R_{j+2,k}(\rho_{k}(u)-\epsilon,u) = 1 - 4(k-j-1)\rho_{k}(u)^2u -2\rho_{k}(u) + 2\rho_{k}(u) \sqrt{ \gamma_{j+1}(u)}\sqrt{\epsilon} + \mathcal{O}(|\epsilon|). \]
We set $a_{j+2}(u):=1 - 4(k-j-1)\rho_{k}^2(u)u-2\rho_{k}(u)$ and $b_{j+2}(u):=2\rho_{k}(u) \sqrt{\gamma_{j+1}(u)}$. Now we proceed by induction. 
Assume $R_{i,k}(\rho_{k}(u)-\epsilon,u)= a_i(u)+b_i(u) \sqrt{\epsilon}+\mathcal{O}(|\epsilon|)$.
We have just checked that it holds for $i=j+2$.
Now we perform the induction step $i \mapsto i+1$.

Using the recursion (\ref{radicalsrecursion}) for $R_{i,k}$ and plugging in the expansion $a_i(u)+b_i(u) \sqrt{\epsilon}+\mathcal{O}(|\epsilon|)$ for $R_{i,k}(\rho_{k}(u)-\epsilon,u)$ yields
\[ R_{i+1,k}(\rho_{k}(u)-\epsilon,u)=1-4(k-i)\rho_{k}(u)^2u - 2\rho_{k}(u) +2\rho_{k}(u) \sqrt{a_i(u)} + \frac{b_i(u) \rho_{k}(u)}{\sqrt{a_i(u)}} \sqrt{\epsilon} + \mathcal{O}(|\epsilon|). \] 
Setting $\forall i \geq j+2\ \ a_{i+1}(u):=1-4(k-i)\rho_{k}^2(u)u - 2\rho_{k}(u) +2\rho_{k}(u) \sqrt{a_i(u)}$ and $b_{i+1}(u):=\frac{b_i(u) \rho_{k}(u)}{\sqrt{a_i(u)}},$
we obtain
\[ R_{i+1,k}(\rho_{k}(u)-\epsilon,u)= a_{i+1}(u)+b_{i+1}(u)\sqrt{\epsilon} + \mathcal{O}(|\epsilon|). \]
Expanding $b_i(u)$, using its recursive relation and $b_{j+2}(u)=2\rho_{k}(u) \sqrt{\gamma_{j+1}(u)}$ we get for $i > j+1$
\[ b_i(u) = \frac{2\rho_{k}^{i-j}(u)\sqrt{\gamma_{j+1}(u)}}{\prod_{l=j+1}^{i-1} \sqrt{a_l(u)}}. \]
\end{itemize}
\end{proof}

We know that for sufficiently large $i$ the sequence $u_i$ is given by $u_i = \lfloor \chi^{2^i} \rfloor$, with $\chi \approx 1.36660956 \ldots$ (see \cite[Lemma 18]{bodini2015number}). Therefore we have $N_j \sim u_j^2 \sim \chi^{2^j 2}$ and $N_j<k<N_{j+1}=O(N_j^2)$, which gives $j \asymp \log \log k$. This implies that $j+1<k+1$, \textit{i.e.}, that the dominant singularity $\rho_{k}(u)$ cannot come from the outermost radical.

\begin{rem}
Obviously the same is true for the case $k=N_j$. Thus, the dominant singularity never comes from the outermost radical.
\end{rem}

Using Proposition \ref{prop:exprjkaroundsing} and $H_{k}(z,u)=\frac{1}{2z}(1-\sqrt{R_{k+1,k}(z,u)})$ we get

\[ H_{k}(\rho_{k}(u)-\epsilon,u)= \frac{1-\sqrt{a_{k+1}(u)}}{2\rho_{k}(u)}-\frac{b_{k+1,k}(u)}{4\rho_{k}(u)\sqrt{a_{k+1}(u)}} \sqrt{\epsilon} + \mathcal{O}(|\epsilon|), \]

which yields
\begin{align}
\label{equ_znhkzutotal}
 [z^n]H_{k}(z,u) = h_k(u) \rho_{k}(u)^{-n} \frac{n^{-\frac{3}{2}}}{\Gamma(-\frac{1}{2})} \left( 1+ O \left( \frac{1}{\sqrt{n}} \right) \right), 
 \end{align}
with 
\[ h_k(u) = -\frac{b_{k+1}(u)\sqrt{\rho_k(u)}}{4\rho_{k}(u)\sqrt{a_{k+1}(u)}}. \]

Taking a look at the recursive definitions of $a_{i}(u)$ and $b_i(u)$ (see Proposition \ref{prop:exprjkaroundsing}), it can easily be seen that these functions are not equal to zero in a neighbourhood of $u=1$. 
We know that $a_{j+2}(1)$ is positive, since
\begin{align*}
a_{j+2}(1)=1- 4(k-j-1)\rho_k(1)^2-2\rho_k(1)= 1- 4(k-j)\rho_k(1)^2-2\rho_k(1) + 4\rho_k^2,
\end{align*}
and $1- 4(k-j)\rho_k(1)^2-2\rho_k(1)>0$ (see \cite{bodini2015number}).
By induction we can show that the sequence $a_i:=a_i(1)$ is monotonically increasing. Let us assume that $a_{i-1}<a_i$, then we get
\begin{align*}
a_{i+1}&>1-4(k-i)\rho_k(1)^2-2\rho_k(1)+2\rho_k(1)\sqrt{a_i} \\
&> 1- 4(k-i+1)\rho_k(1)^2-2\rho_k(1)+2\rho_k(1)\sqrt{a_{i-1}} + 4\rho_k(1)^2 > a_i + 4\rho_k(1)^2.
\end{align*}
It is obvious that if $b_{j+2}:=b_{j+2}(1)$ is non-zero, than all the $b_i$'s, which are defined by
\begin{align*}
b_i = \frac{\rho_k(1)b_{i-1}}{a_{i-1}},
\end{align*}
are non-zero. In order to prove that $b_{j+2}=2\rho_k(1) \sqrt{-\frac{\partial}{\partial z} R_{j+1,k}(\rho_k(1),1)}$ is non-zero, we also proceed by induction. Since
\begin{align*}
R_{1,k}(z,1)=1-4kz^2,
\end{align*}
we can see that $\frac{\partial}{\partial z} R_{1,k}(\rho_k(1),1) <0$, and assuming $\frac{\partial}{\partial z} R_{i,k}(\rho_k(1),1) < 0$ and using
\begin{align*}
\frac{\partial}{\partial z} R_{i+1,k}(z,1)= -8(k-i)z-2+2\sqrt{R_{i,k}(z,1)}+\frac{z}{\sqrt{R_{i,k}(z,1)}} \frac{\partial}{\partial z} R_{i}(z,1),
\end{align*}
we proved that all $b_i$'s are non-zero.
Thus, we get that $h_k(u) \neq 0$.

\medskip
Using (\ref{equ_znhkzutotal}) and Theorem \ref{theo:asymphleqkueg1} we get for $n \longrightarrow \infty$
\begin{align}
\label{equ:znhleqkzuznhleqkz1}
 \frac{[z^n]H_{k}(z,u)}{[z^n]H_{k}(z,1)} = \frac{h_k(u)}{h_k \Gamma(-1/2)} \left( \frac{\rho_{k}(1)}{\rho_{k}(u)} \right)^n \left( 1+ O \left( \frac{1}{n} \right) \right).
\end{align}

Assuming that $\sigma^2=B''(1)+B'(1)-B'(1)^2\neq 0$ with $B(u)=
\frac{\rho_{k}(1)}{\rho_{k}(u)}$ we can apply the Quasi-Power Theorem. As stated in Section \ref{ch:mainresults} the proof of this assumption appears to be quite difficult, since there is only very little known about the function $\rho_{k}(u)$. However, it seems very likely that this condition will be
fulfilled for arbitrary $k \in (N_j, N_{j+1})$, so that the Quasi-Power Theorem can be applied and
we get that the number of leaves in lambda-terms with bounded number of De Bruijn levels is asymptotically normally distributed with asymptotic mean $\mu n$ and variance $\sigma^2 n$, respectively, where $\mu = B'(1)$ and $\sigma^2=B''(1)+B'(1)-B'(1)^2$, with $B(u)= \frac{\rho_{k}(1)}{\rho_{k}(u)}$.

\begin{table}[h!]
\center
\small
\begin{tabular}{c|c|c|c}
bound $k$ & $j+1$ & $B''(1)+B'(1)-B'(1)^2$ & $B'(1)$ \\ \hline
\bf{1}	& \bf{2} & \bf{0} & \bf{0} \\
2	& 2 & 0.0385234386	 & 0.4381229337 \\
3	& 2 & 0.0210625856	 & 0.4414407371 \\
4	& 2 & 0.0167136805	 & 0.4463973717 \\
5	& 2 & 0.0148700270	 & 0.4504258849 \\
6	& 2 & 0.0138224393	 & 0.4536185043 \\
7	& 2 & 0.0131157948	 & 0.4561987871 \\
\bf{8}	& \bf{3} & \bf{0.0125868052} & \bf{0.4583333333} \\
9	& 3 & 0.0582322465	 & 0.4566104777 \\
10	& 3 & 0.0470481360	 & 0.4560418340 \\
11	& 3 & 0.0396601986	 & 0.4560810348 \\
12	& 3 & 0.0345090124	 & 0.4564489368 \\
\vdots & \vdots & \vdots & \vdots \\
133 & 3 & 0.0077469541 & 0.4821900098 \\
134 & 3 & 0.0077234960 & 0.4822482745 \\
\bf{135} & \bf{4} & \bf{0.0077002803} & \bf{0.4823059361} \\
136 & 4 & 0.0132855719 & 0.4823515285 \\
137 & 4 & 0.0131816901 & 0.4823968564 \\
138 & 4 & 0.0130800422 & 0.4824419195 \\
139 & 4 & 0.0129805564 & 0.4824867175
 \end{tabular}
\vspace{2mm}
\caption{Table summarizing the coefficients occurring in the variance and the mean for some initial values of $k$.}
\label{tab:initialvalues}
\end{table}

\subsection{The case $k = N_j$}
\label{subsection:k=njtotal}

We know from \cite{bodini2015number} that in the case $k = N_j$ both radicands $R_{j,k}(z,1)$ and $R_{j+1,k}(z,1)$ vanish simultaneously and the dominant singularity is therefore of type $\frac{1}{4}$.
This is not true for the radicands $R_{j,k}(z,u)$ and $R_{j+1,k}(z,u)$ when $u$ is in a neighbourhood of 1. Thus, we have a discontinuity at $\rho_k(1)$, which is why we do not get any uniform expansions of the radicands in a neighbourhood of $\rho_k(1)$.

In order to overcome this problem we will set $u=1+\epsilon$ and investigate how the radicands behave in a neighbourhood of the dominant singularity $\rho_k(u)=\rho_k(1+\epsilon)$. Subsequently we will use the abbreviation $\rho_k:=\rho_k(1)$.

\begin{lemma}
\label{lem:domsingvonjradicand}
For $u=1+\epsilon$ with $\epsilon \longrightarrow 0$ so that $\epsilon \in \mathbb{C} \setminus \mathbb{R}^{-}$, the dominant singularity $\rho_k(u)=\rho_k(1+\epsilon)$ of the bivariate generating function $H_k(z,1+\epsilon)$ comes from the $j$-th radicand $R_{j,k}(z,u)$.
\end{lemma}

\begin{proof}
Setting $u=1+\epsilon$,expanding $\rho_k(u)$ around $1$ and plugging into the recursive definition of the radicands yields \begin{align*}
R_{j,k} \( \rho_k(1+\epsilon),1+\epsilon \) = &1-4(k-j+1)(\rho_k^2+2\rho_k\rho_k'\epsilon+ (\rho_k'^2+2\rho_k\rho_k'')\epsilon^2 +\rho_k^2\epsilon) \\  &- (2\rho_k+2\rho_k'\epsilon +2\rho_k''\epsilon^2) \(1-\sqrt{R_{j-1,k}\( \rho_k(1+\epsilon),1+\epsilon \)} \) + \mathcal{O} ( |\epsilon| ).
\end{align*}
Using $1-4(k-j)\rho_k^2-2\rho_k=0$ and $\sqrt{R_{j-1,k}\( \rho_k(1+\epsilon),1+\epsilon \)} = \sqrt{R_{j-1,k}\( \rho_k,1 \) + \mathcal{O}(|\epsilon|)} = 2\rho_k + \mathcal{O}(|\epsilon|)$, which are both shown in \cite{bodini2015number}, we get
\begin{align*}
R_{j,k} \( \rho_k(1+\epsilon),1+\epsilon \) = &-4(k-j+1)(2\rho_k\rho_k'\epsilon+ (\rho_k'^2+2\rho_k\rho_k'')\epsilon^2) \\ &- (2\rho_k'\epsilon +2\rho_k''\epsilon^2) \(1-2\rho_k + \mathcal{O}(|\epsilon|) \) + \mathcal{O} ( |\epsilon| ).
\end{align*}
Thus, $R_{j,k} \( \rho_k(1+\epsilon),1+\epsilon \)= \Theta(|\epsilon|)$.

Using this result and again the recursive definition of the radicands results in
\begin{align*}
R_{j+1,k} \( \rho_k(u),1+\epsilon \)= 2\sqrt{R_{j,k}(\rho_k(u),1+\epsilon)} + \mathcal{O} \( |\epsilon| \) = \Theta(\sqrt{|\epsilon|}).
\end{align*}

Thus, we see that $|R_{j+1,k} (\rho_k(u),u)| \gg |R_{j,k} (\rho_k(u),u)|$ in a neighbourhood of $u=1$, which implies that the dominant singularity has to come from the $j$-th radicand, \textit{i.e.} $R_{j,k}(\rho_k(u),u)=0$ for $u$ being sufficiently close to $1$.
\end{proof}

Now that we know that in this case ($k=N_j$) the dominant singularity of $H_k(z,u)$ in a neighbourhood of $u=1$ comes from the $j$-th radicand, we investigate the expansions of the radicands thoroughly for $u=1+\frac{s}{\sqrt{n}}$ in a neighbourhood with radius $\frac{t}{n}$, where $s$ and $t$ are both bounded complex numbers (\textit{cf}. Figure \ref{fig:proofidea}).

\begin{lemma}
\label{lem:exprjrj+1rho1}
Let $z=\rho_k(u)=\rho_k(1+\frac{s}{\sqrt{n}})$ be the dominant singularity of the bivariate generating function $H_k(z,1+\frac{s}{\sqrt{n}})$ with bounded $s \in \mathbb{C}$. Then, as $n \longrightarrow \infty$,
\begin{itemize}
	\item[(i)] $R_{j,k} \Big(\rho_k(u)\(1+\frac{t}{n}\),1+\frac{s}{\sqrt{n}}\Big) = \frac{1}{n} p_j(t) + \mathcal{O} \( \frac{1}{n^{3/2}} \)$, \newline
with $p_j(t):=-8t(k-j+1)\rho_k^2 - 2\rho_kt+4\rho_k^2t+2t\rho_kf(\frac{t}{n})$ where $f(\frac{t}{n})$ is an analytic function around 0;\\
	\item[(ii)] $R_{j+1,k}\Big(\rho_k(u)\(1+\frac{t}{n}\),1+\frac{s}{\sqrt{n}}\Big) =  \frac{1}{\sqrt{n}} p_{j+1}(s,t) + \mathcal{O} \( \frac{1}{n} \),$ \newline
where $p_{j+1}(s,t)=2\rho_k \sqrt{p_j(t)}-4(k-j)(2\rho_k\rho_k's+\rho_k^2s)-2\rho_k's$;\\
	\item[(iii)] $R_{i,k}\Big(\rho_k(u)\(1+\frac{t}{n}\),1+\frac{s}{\sqrt{n}}\Big) = \hat{C}_{i} + \frac{1}{\sqrt[4]{n}} p_{i}(s,t) + O \( \frac{1}{\sqrt{n}} \) \ \ \ \ \text{for} \ i \geq j+2$, \newline
where $\hat{C}_{i}$ are constants and $p_{i}(s,t)$ analytic functions $s$ and $t$.
\end{itemize}
\end{lemma}

\begin{proof}
We start with setting $u=1+\frac{s}{\sqrt{n}}$ and $z=\rho_k(u)(1+\frac{t}{n})$ with bounded $s,t \in \mathbb{C}$ (\textit{cf.} Figure \ref{fig:proofidea}), which results in
\begin{align*}
&R_{j+1,k}\(\rho_k(u)\(1+\frac{t}{n}\),1+\frac{s}{\sqrt{n}}\)= \\
&1-4(k-j)\rho_k(u)^2\(1+\frac{t}{n}\)^2\(1+\frac{s}{\sqrt{n}}\)-2\rho_k(u)\(1+\frac{t}{n}\)\(1- \sqrt{R_{j,k}}\),\\ 
\text{and}\\
&R_{j,k}\(\rho_k(u)\(1+\frac{t}{n}\),1+\frac{s}{\sqrt{n}}\)=\\
&1-4(k-j+1)\rho_k(u)^2\(1+\frac{t}{n}\)^2\(1+\frac{s}{\sqrt{n}}\)-2\rho_k(u)\(1+\frac{t}{n}\)\(1- \sqrt{R_{j-1,k}} \), 
\end{align*}
%
where the radicand in the square root in the last bracket of both equations is of course also evaluated at $(z,u)=\(\rho_k(1+\frac{s}{\sqrt{n}})(1+\frac{t}{n}),1+\frac{s}{\sqrt{n}}\)$, but we will omit this notation from now on to ensure a simpler reading, \textit{\textit{i.e.}}, subsequently we will write $R_{i,k}$ instead of $R_{i,k}\(\rho_k(1+\frac{s}{\sqrt{n}})\(1+\frac{t}{n}\),1+\frac{s}{\sqrt{n}}\)$.

Expanding $\rho_k(1+\frac{s}{\sqrt{n}})$ around 1 and using the recursive definition for the radicands yields
\begin{align} 
\label{equ:rjkasympnearueq1}
\begin{aligned}
R_{j,k} &= 1-4(k-j+1) \( \rho_k^2 + 2\rho_k \rho_k' \frac{s}{\sqrt{n}} + (\rho_k'^2+2\rho_k\rho_k'') \frac{s^2}{n} + \rho_k^2 \frac{s}{\sqrt{n}} + 2\rho_k \rho_k' \frac{s^2}{n} + \rho_k^2 \frac{2t}{n} \) \\
&-2\( \rho_k + \rho_k' \frac{s}{\sqrt{n}} + \rho_k'' \frac{s^2}{2n} + \rho_k \frac{t}{n} \) \( 1- \sqrt{R_{j-1,k}}\) + \mathcal{O} \( \frac{1}{n^{3/2}} \). 
\end{aligned}
\end{align}

From Lemma \ref{lem:domsingvonjradicand} we know that for $u$ in a sufficiently small vicinity of 1 the dominant singularity of $H_k(z,u)$ comes from the $j$-th radicand, \textit{i.e.} $R_{j,k}\( \rho_k(u),u \)=0$.
Expanding $R_{j,k}\( \rho_k(1+\frac{s}{\sqrt{n}}),1+\frac{s}{\sqrt{n}} \)$ this yields
\begin{align*}
&1-4(k-j+1)\( \rho_k^2 + 2\rho_k \rho_k' \frac{s}{\sqrt{n}} + (\rho_k'^2+2\rho_k\rho_k'') \frac{s^2}{n} + \rho_k^2 \frac{s}{\sqrt{n}} + 2\rho_k \rho_k' \frac{s^2}{n} \) \\
&-2\( \rho_k + \rho_k' \frac{s}{\sqrt{n}} + \rho_k'' \frac{s^2}{2n} \) \( 1- \sqrt{R_{j-1,k}\( \rho_k \( 1+\frac{s}{\sqrt{n}} \),1+\frac{s}{\sqrt{n}} \)}\) + \mathcal{O} \( \frac{1}{n^{3/2}} \)=0.
\end{align*}

Thus, Equation (\ref{equ:rjkasympnearueq1}) simplifies to
\begin{align}
\label{equ:rjkasympfertig}
R_{j,k} = -4(k-j+1) \rho_k^2 \frac{2t}{n} - 2\rho_k \frac{t}{n} + 4\rho_k^2 \frac{t}{n} + 2\rho_k\frac{t}{n}f\(\frac{t}{n}\) + \mathcal{O} \( \frac{1}{n^{3/2}} \),
\end{align}
where $\frac{t}{n}f\(\frac{t}{n}\)=\sqrt{R_{j-1,k}}-\sqrt{R_{j-1,k}\( \rho_k(1+\frac{s}{\sqrt{n}}),1+\frac{s}{\sqrt{n}}\)}$, where $f\(\frac{t}{n}\)$ is analytic around 0.

Therefore, the proof of $(i)$ is finished.

Proceeding equivalently for $R_{j+1,k}$ results in
\begin{align*}
R_{j+1,k} &= \frac{1}{\sqrt{n}} \Big( -4(k-j)(2\rho_k\rho_k's+\rho_k^2s)-2\rho_k's \Big) + 2\rho_k\sqrt{R_{j,k}} + O \( \frac{1}{n} \).
\end{align*}
Inserting Equation (\ref{equ:rjkasympfertig}) for $R_{j,k}$ we proved the second statement of the lemma.

Going one step further leads to
\[ R_{j+2,k} = \hat{C}_{j+2} + \frac{1}{\sqrt[4]{n}} p_{j+2}(s,t) + O \( \frac{1}{\sqrt{n}} \), \]
with $\hat{C}_{j+2}:=4\rho_k^2$ and $p_{j+2}(s,t):= 2\rho_k\sqrt{p_{j+1}(s,t)}$, where $p_{j+1}(s,t)$ is defined as in Lemma \ref{lem:exprjrj+1rho1}.

Now we proceed by induction.
Therefore we assume that $R_{i,k} = \hat{C}_{i} + \frac{1}{\sqrt[4]{n}} p_{i}(s,t) + O \( \frac{1}{\sqrt{n}} \)$ with $i \geq j+2$.
Thus, we get
\begin{align*}
\begin{aligned}
R_{i+1,k} &= 1-4(k-i) \( \rho_k^2 + 2\rho_k \rho_k' \frac{s}{\sqrt{n}} + (\rho_k'^2+2\rho_k\rho_k'') \frac{s^2}{n} + \rho_k^2 \frac{s}{\sqrt{n}} + 2\rho_k \rho_k' \frac{s^2}{n} + \rho_k^2 \frac{2t}{n} \) \\
&-2\( \rho_k + \rho_k' \frac{s}{\sqrt{n}} + \rho_k'' \frac{s^2}{2n} + \rho_k \frac{t}{n} \) \( 1- \sqrt{R_{i,k}}\) + \mathcal{O} \( \frac{1}{n^{3/2}} \). 
\end{aligned}
\end{align*}
Inserting the induction hypothesis and simplifying yields
\[ R_{i+1,k} = 4(i-j)\rho_k^2+2\rho_k \sqrt{\hat{C}_i} + \frac{1}{\sqrt[4]{n}} \frac{\rho_kp_i(s,t)}{\sqrt{\hat{C}_i}} + \mathcal{O} \( \frac{1}{\sqrt{n}} \). \]
Setting $\hat{C}_{i+1}:= 4(i-j)\rho_k^2+2\rho_k \sqrt{\hat{C}_i}$ and $p_{i+1}(s,t):=\frac{\rho_k}{\sqrt{\hat{C}_i}} p_i(s,t)$ completes the proof.
\end{proof}

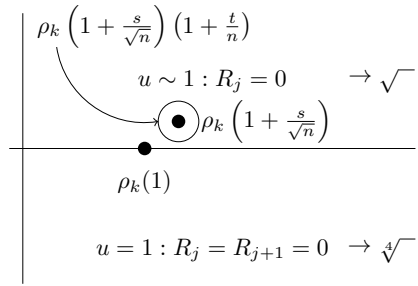
\begin{figure}[h!]
\scalebox{0.9}{\begin{tikzpicture}
	\draw (0,2) -- (6,2);
	\draw (0.2,0) -- (0.2,4);
	\fill (2,2) circle (0.1);
	\node (k) at (2,1.5) {$\rho_k(1)$};
	\node (k) at (3,0.5) {$u=1: R_j=R_{j+1}=0$};
	\fill (2.5,2.4) circle (0.1);
	\node (k) at (3,3) {$u \sim 1: R_j=0$};
	\node (k) at (5.5,0.5) {$\rightarrow \sqrt[4]{ \ \ }$};
	\node (k) at (5.5,3) {$\rightarrow \sqrt{ \ \ }$};
	\node (k) at (3.8,2.4) {$\rho_k \(1+\frac{s}{\sqrt{n}}\)$};
	\draw (2.5,2.4) circle (0.3);
	\node (k) at (2,3.8) {$\rho_k \( 1+ \frac{s}{\sqrt{n}} \) \( 1+ \frac{t}{n} \)$};
	\draw [->, bend angle=45, bend right]  (0.7,3.5) to (2.2,2.4);
	\end{tikzpicture}}
\caption{Sketch of the idea of the proof.}
\label{fig:proofidea}
\end{figure}

\begin{prop}
\label{prop:znhkzutotal}
Let $H_k(z,u)$ be the bivariate generating function of the class of closed lambda-terms with at most $k$ De Bruijn levels. Then the $n$-th coefficient of $H_k(z,1+\frac{s}{\sqrt{n}})$ with bounded $s \in \mathbb{C}$ is given by
\begin{align*}
[z^n] H_{k}(z,1+\frac{s}{\sqrt{n}}) = C_k(s) \rho_k^{-n} n^{-\frac{5}{4}} \( 1+ O \( n^{-\frac{3}{4}} \) \), \ \ \ \text{as} \ n \longrightarrow \infty,
\end{align*}
with a constant $C_k(s) \neq 0$.
\end{prop}
\begin{proof}
Let us remember that $H_{k}(z,1+\frac{s}{\sqrt{n}}) = \frac{1- \sqrt{R_{k+1,k}(z,1+\frac{s}{\sqrt{n}})}}{2z}$. Thus, with the well-known Cauchy coefficient formula we get
\begin{align*} [z^n]H_{k} \( z, 1+\frac{s}{\sqrt{n}} \) &= \frac{1}{2i\pi} \int_{\gamma} \frac{H_{k} \( z, 1+\frac{s}{\sqrt{n}} \)}{z^{n+1}} dz \\
&= \frac{1}{2i\pi} \int_{\gamma} \frac{1-\sqrt{R_{k+1,k} \( z, 1+\frac{s}{\sqrt{n}}\)} }{2z^{n+2}} dz,
\end{align*}
where $\gamma$ encircles the dominant singularity $\rho_k(u)$ as depicted in Figure \ref{fig:integrationcontours}. We denote the small Hankel-like part of the integration contour $\gamma$ that contributes the main part of the asymptotics by $\gamma_H$ (\textit{cf}. Figure \ref{fig:integrationcontours}). The curve $\gamma_H$ encircles $\rho_k(u)$ at a distance $\frac{1}{n}$ and its straight parts (that lead into the direction $\rho_k(u) \cdot \infty$) have the length $\frac{\log^2(n)}{n}$. On $\gamma \setminus \gamma_H$ we have $|z|=|\rho_k(u)| \Big|1+ \frac{\log^2(n)}{n} + \frac{i}{n} \Big|$. Thus, using the transformation $z=\rho(u)\( 1+ \frac{t}{n} \)$, which changes $\gamma_H$ to $\tilde{\gamma}_H$, and Lemma \ref{lem:exprjrj+1rho1} and estimating the contribution of $\gamma \setminus \gamma_H$ implies that there exisits a $K>0$ such that

\begin{align}
[z^n]H_{k} \( z, 1+\frac{s}{\sqrt{n}} \) &= \frac{1}{2i\pi} \int_{\tilde{\gamma}_H} \frac{1-\sqrt{\hat{C}_{k+1}+\frac{1}{\sqrt[4]{n}}p_{k+1}(s,t)+\mathcal{O} \( \frac{1}{\sqrt{n}} \)}}{2\rho^{n+1}e^tn} dt + \mathcal{O} \( e^{-K\log^2(n)} \)\\
\label{equ:znhkmitintegral}
&= \frac{1}{2i\pi} \int_{\tilde{\gamma}_H} \frac{1-\sqrt{\hat{C}_{k+1}} - \frac{1}{2\sqrt[4]{n}\sqrt{\hat{C}_{k+1}}}p_{k+1}(s,t)+\mathcal{O} \( \frac{1}{\sqrt{n}} \)}{2\rho_k^{n+1}e^tn} dt+\mathcal{O} \( e^{-K\log^2(n)} \).
\end{align}

\begin{figure}[h!]
  \centering
  \scalebox{0.6}{
	\begin{tikzpicture}
	\draw [domain=57:394] plot ({2+2.5*cos(\x)}, {3+2.5*sin(\x)});
	\fill (3,4) circle (0.1);
	\draw (0,3) -- (6,3);
	\draw (2,2) -- (2,5);
	\draw (2.7,4.4) -- (3.4,5.1);
	\draw (3.35,3.7) -- (4.05,4.4);
	\draw (2.7,4.4) to[bend right=50] (2.7,3.7);
	\draw (3.35,3.7) to[bend left=50] (2.7,3.7);
	\coordinate[label=-45:$\gamma_H$] (z) at (3.5,4);
	\coordinate[label=-45:$\gamma$] (z) at (-0.2,5.3);
    	\coordinate[label=-45:$\rho_k(u)$] (z) at (2.8,4.7);
	\end{tikzpicture}}
	\hspace{1.1cm}
  \scalebox{0.8}{
	\begin{tikzpicture}
	\fill (2,3) circle (0.1);
	\draw (0,3) -- (6,3);
	\draw (2,2) -- (2,5);
	\draw (2,3.5) -- (6,3.5);
	\draw (2,2.5) -- (6,2.5);
	\draw (2,3.5) to[bend right=40] (1.5,3);
	\draw (2,2.5) to[bend left=40] (1.5,3);
    	\coordinate[label=-45:$\mathcal{H}$] (z) at (3.5,4);
    	\end{tikzpicture}}
	\caption{The integral contours $\gamma$ and $\mathcal{H}$.}
\label{fig:integrationcontours}
\end{figure}
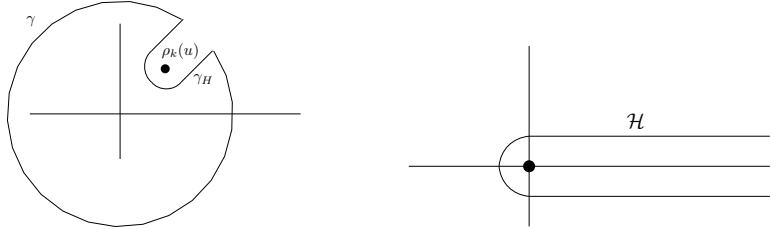

Now, let us observe how the function $p_{k+1}(s,t)$ looks like by using the recursive definition $p_{i+1}(s,t)= \frac{\rho_k}{\sqrt{\hat{C}_i}} p_i(s,t)$ and $p_{j+2}(s,t)=2\rho_k \sqrt{2\rho_k\sqrt{p_j(t)}+q(s)}$, with a polynomial $q(s)=s\(-4(k-j)(2\rho_k\rho_k'+\rho_k^2)-2\rho_k'\)$ that is linear in $s$. Thus, $p_{k+1}(s,t)=D \cdot p_{j+2}(s,t)$ with a constant $D$.
Inserting this into (\ref{equ:znhkmitintegral}) and splitting the integral yields
\begin{align*}
&[z^n]H_{k} \( z, 1+\frac{s}{\sqrt{n}} \) = \\
&\frac{\rho_k^{-n}}{4i\pi\rho_kn} \( \int_{\tilde{\gamma}_H} \(1-\sqrt{\hat{C}_{k+1}}\)e^{-t}dt - \int_{\tilde{\gamma}_H} \frac{De^{-t}}{2\sqrt[4]{n}\sqrt{\hat{C}_{k+1}}} \sqrt{2\rho_k\sqrt{p_j(t)}+q(s)}dt + \int_{\tilde{\gamma}_H} \mathcal{O} \( \frac{1}{\sqrt{n}} \) e^{-t}dt   \).
\end{align*}

The first integral is zero and the third integral contributes $\mathcal{O} \( \frac{1}{\sqrt{n}} \)$. Thus, the main part of the asymptotics results from the second integral: There are some constants $A(s)$ and $B(s)$ such that
\begin{align*}
- \int_{\tilde{\gamma}_H} \frac{De^{-t}}{\sqrt[4]{n}\sqrt{\hat{C}_{k+1}}}\sqrt{2\rho_k\sqrt{p_j(t)}+q(s)}dt &= 
- \int_{\tilde{\gamma}_H} \frac{De^{-t}}{\sqrt[4]{n}\sqrt{\hat{C}_{k+1}}}\sqrt{A(s)t+B(s)+\mathcal{O}\(\frac{\log^4(n)}{n}\)}dt\\
&= - \int_{\tilde{\gamma}_H} \frac{De^{-t}}{\sqrt[4]{n}\sqrt{\hat{C}_{k+1}}}\sqrt{A(s)t+B(s)}dt + \mathcal{O} \( \frac{\log^6(n)}{n} \) \\
&= - \int_{\mathcal{H}} \frac{De^{-t}}{\sqrt[4]{n}\sqrt{\hat{C}_{k+1}}}\sqrt{A(s)t+B(s)}dt + \mathcal{O} \( e^{-\tilde{K}\log^2(n)} \) \\
&\sim \tilde{C}(s)\frac{1}{\sqrt[4]{n}}.
\end{align*}
Here $\tilde{K}$ denotes a suitable positive constant, and 
$\mathcal{H}$ denotes the classical Hankel curve, \emph{i.e.}, the noose-shaped curve that winds
around 0 and starts and ends at $+\infty$ (\textit{cf}. Figure \ref{fig:integrationcontours}).

Finally, using this result we get
\begin{align*}
[z^n]H_{k} \( z, 1+\frac{s}{\sqrt{n}} \) = C(s) \rho_k \( 1+\frac{s}{\sqrt{n}} \)^{-n} n^{-5/4} \( 1+ \mathcal{O} \( \frac{1}{\sqrt[4]{n}} \) \), \ \ \ \text{for} \ n \longrightarrow \infty,
\end{align*}
with a constant $C(s)$ that depends on $s$.
\end{proof}

Now we show that the characteristic function of our standardized sequence of random variables tends to the characteristic function of the normal distribution.
\begin{lemma}
Let $X_n$ be the total number of variables in a random lambda-term with at most $k$ De Bruijn levels. Set $\sigma^2 :=2 \cdot \( \frac{\rho_k'(1)}{\rho_k(1)}-\frac{\rho_k''(1)}{\rho_k(1)}+\frac{\rho_k'(1)^2}{\rho_k(1)^2} \)$.  If $\sigma^2 \neq 0$, then 
\[ Z_n = \frac{X_n - \mathbb{E}X_n}{\sqrt{n}} \longrightarrow \mathcal{N} \( 0, \sigma^2 \). \]
\end{lemma}

\begin{proof}
For the standardised sequence of random variables $Z_n$ we have with $\mu :=\frac{\mathbb{E}X_n}{n}$
\[ Z_n = \frac{X_n - \mathbb{E}X_n}{\sqrt{n}} = \frac{X_n}{\sqrt{n}}-\mu \sqrt{n}. \]
Its characteristic function reads as
\begin{align} \phi_{Z_n}(s) &= \mathbb{E}(e^{isZ_n}) = e^{-is\mu \sqrt{n}} \phi_{X_n}\( \frac{s}{\sqrt{n}} \) =  e^{-is\mu \sqrt{n}} \mathbb{E}(e^{\frac{isX_n}{\sqrt{n}}}) = \\
&= e^{-is\mu \sqrt{n}} \frac{[z^n] H_k(z,e^{\frac{is}{\sqrt{n}}})}{[z^n]H_k(z,1)}.
\end{align}
From Proposition \ref{prop:znhkzutotal} we know 
\begin{align*}
\frac{[z^n] H_k(z,1+\frac{s}{\sqrt{n}})}{[z^n]H_k(z,1)} \sim C(s) \( \frac{\rho_k(1+\frac{s}{\sqrt{n}})}{\rho_k(1)} \) ^{-n},
\end{align*}
where the constant $C(s) \sim 1$ for $n \longrightarrow \infty$.

Thus,
\begin{align*}
\phi_{Z_n}(s) &= e^{-is\mu \sqrt{n}} \frac{[z^n] H_k(z,e^{\frac{is}{\sqrt{n}}})}{[z^n]H_k(z,1)} \sim e^{-is\mu \sqrt{n}} \( \frac{\rho_k \(1+ \frac{si}{\sqrt{n}} - \frac{s^2}{n} + \mathcal{O}\( \frac{|s^3|}{n^{3/2}} \) \)}{\rho_k(1)} \) ^{-n} =\\
&= e^{-is\mu \sqrt{n}} \exp \( -n \cdot \( \log\(1+ \frac{\rho_k'is}{\rho_k\sqrt{n}} - \frac{s^2}{n} \frac{\rho_k'}{\rho_k}- - \frac{s^2}{n} \frac{\rho_k''}{\rho_k} \) + \mathcal{O}\( \frac{|s^3|}{n^{3/2}} \) \) \) \\
&\sim e^{-is\mu \sqrt{n}} e^{-is\sqrt{n} \frac{\rho_k'}{\rho_k}} e^{s^2 \( -\frac{\rho_k'}{\rho_k}+\frac{\rho_k''}{\rho_k}-\frac{\rho_k'^2}{\rho_k^2} \)}.
\end{align*}
Since we know that the expected value of the standardised random variable is zero, we get $\mu=-\frac{\rho_k'(1)}{\rho_k(1)}+o \( \frac{1}{\sqrt{n}}\)$,
and thus
\begin{align*}
\phi_{Z_n}(s) \sim e^{- \frac{s^2 \sigma^2}{2}},
\end{align*}
with $\sigma^2=2 \cdot \( \frac{\rho_k'(1)}{\rho_k(1)}-\frac{\rho_k''(1)}{\rho_k(1)}+\frac{\rho_k'(1)^2}{\rho_k(1)^2} \)$,
which completes the proof.
\end{proof}

Thus, we get that the total number of leaves in lambda-terms with a bounded number of De Bruijn levels is asymptotically normally distributed.

\section{Unary profile of lambda-terms with bounded number of De Bruijn levels} 
\label{ch:levelheight}

\subsection{Leaves}
\label{sec:leaveslevels}

The aim of this section is the investigation of the distribution of the number of leaves in the different De Bruijn levels in closed lambda-terms with bounded number of De Bruijn levels. In order to do so, let us consider that each De Bruijn level in such a lambda-term corresponds to one or more binary trees that contain different types of leaves, where the number of types corresponds to the respective level (\textit{cf.} Figure \ref{fig:unarylevels}), \textit{i.e.}, in the $i$-th De Bruijn level there may be $i$ different types of leaves.
Let $\mathcal{C}$ be the class of binary trees. Using the notation from the previous sections we can specify this class by
\[ \mathcal{C}= \mathcal{Z}  + ( \mathcal{A} \times \mathcal{C} \times \mathcal{C}). 
\]

Translating into bivariate generating functions $C(z,u)$ with $z$ marking the size (\textit{i.e.}, the total number of nodes) and $u$ marking the number of leaves, yields $ C(z,u)=\frac{1- \sqrt{1-4uz^2}}{2z}$.

Let $_{k-l}\tilde{H}_{k}(z,u)$ be the generating function of closed lambda-terms with at most $k$ De Bruijn levels, where $z$ marks the size and $u$ marks the number of leaves on the $(k-l)$-th unary level ($0 \leq l \leq k$).
Then we have
\[ _{k-l}\tilde{H}_{k}(z,u) = C(z,C(z,1+ \ldots +C(z,(k-l) \cdot u+ \ldots + C(z, (k-1) +C(z,k))) \ldots) \ldots)), \]
which can be written as
\[ _{k-l}\tilde{H}_{k}(z,u) = \frac{1-\sqrt{\tilde{R}_{k+1,k}(z,u)}}{2z}, \]
with 
\[ \tilde{R}_{1,k}(z,u)=1-4z^2k, \]
\[\tilde{R}_{i,k}(z,u)=1-4z^2(k-i+1)-2z+2z\sqrt{\tilde{R}_{i-1,k}(z,u)}, \ \ \ \text{for} \ 2 \leq i \leq k+1, i \neq l+1, \]
and
\[\tilde{R}_{l+1,k}(z,u)=1-4z^2u(k-l)-2z+2z\sqrt{\tilde{R}_{l-1,k}(z,u)}. \]

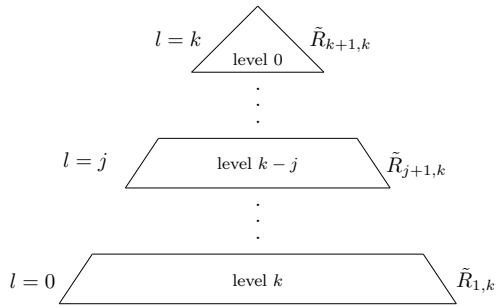
\begin{figure}[h!]
  \centering
\scalebox{0.8}{\begin{tikzpicture}[scale=0.55] 
	\draw (6,9) -- (4,7);
	\draw (6,9) -- (8,7);
	\draw (4,7) -- (8,7);
	\coordinate[label=0: .] (z) at (5.7,6.5);
	\coordinate[label=0: .] (z) at (5.7,6);
	\coordinate[label=0: .] (z) at (5.7,5.5);
	\draw (3,5) -- (9,5);
	\draw (9,5) -- (10,3.5);
	\draw (2,3.5) -- (10,3.5);
	\draw (3,5) -- (2,3.5);
	\coordinate[label=0: .] (z) at (5.7,3);
	\coordinate[label=0: .] (z) at (5.7,2.5);
	\coordinate[label=0: .] (z) at (5.7,2);
	\draw (0,0) -- (12,0);
	\draw (0,0) -- (1,1.5);
	\draw (1,1.5) -- (11,1.5);
	\draw (11,1.5) -- (12,0);
	\node at (3.6,8) {$l=k$};
	\node at (6,7.4) {\footnotesize{level $0$}};
	\node at (6,4.2) {\footnotesize{level $k-j$}};
	\node at (0.8,4.3) {$l=j$};
	\node at (6,0.7) {\footnotesize{level $k$}};
	\node at (-0.8,0.7) {$l=0$};
	\node at (12.6,0.7) {$\tilde{R}_{1,k}$};
	\node at (10.8,4.2) {$\tilde{R}_{j+1,k}$};
	\node at (8.5,8) {$\tilde{R}_{k+1,k}$};
	\end{tikzpicture}}
\caption{A schematic sketch of a lambda-term with at most $k$ De Bruijn levels that exemplifies the notation that is used within this section: If we investigate the number of leaves in the $(k-l)$-th De Bruijn level, for $0 \leq l \leq k$, a factor $u$ is inserted in the recursive definition of the $(l+1)$-th radicand.}
\label{fig:explamationunarylevelsnotation}
\end{figure}

\begin{rem}
Note that the radicands $\tilde{R}_{i,k}$ that are introduced above are very similar to the radicands $R_{i,k}$ that were used in the previous section. The only difference is that now we have a $u$ only in the $(l+1)$-th radicand, while in the previous case $u$ was occurring in all radicands. Thus, from now on we will have further distinctions of cases now depending on the relative position (w.r.t. $l$) of the radicand(s) where the dominant sigularity comes from.
\end{rem}

This chapter consists of two sections. In the first part we will derive the mean values for the number of leaves in the different De Bruijn levels and the second part deals with the distributions of the number of leaves in these levels.

\subsubsection{Mean values}

Now we want to determine the mean for the number of leaves in the different De Bruijn levels, \textit{\textit{i.e.}}
\[ \mathbb{E}X_n= \frac{[z^n] \left( \frac{\partial}{\partial u} \ _{k-l}\tilde{H}_{k}(z,u) \right) |_{u=1}}{[z^n]_{k-l}\tilde{H}_{k}(z,1)}, \]
where $X_n$ denotes the number of leaves in the $(k-l)$-th De Bruijn level of a random closed lambda-term of size $n$ with at most $k$ De Bruijn levels.

In order to do so, we make the following considerations:
\begin{itemize}
\item $\frac{\partial}{\partial u} \tilde{R}_{i,k}(z,u)=0 \ \ \ \ \forall i < l+1$\\

\item$ \frac{\partial}{\partial u} \tilde{R}_{l+1,k}(z,u)= -4z^2(k-l)$ \\

\item $\frac{\partial}{\partial u} \tilde{R}_{i,k}(z,u)= z \cdot \frac{ \frac{ \partial}{ \partial u} \tilde{R}_{i-1,k}(z,u)}{\sqrt{\tilde{R}_{i-1,k}(z,u)}}\ \ \ \ \ \forall i > l+1$
\end{itemize}

Therefore we get
\begin{align}
\label{equ:ablkminuslhleqk}
\left( \frac{\partial}{\partial u} \ _{k-l}\tilde{H}_{k}(z,u) \right) \bigg|_{u=1} = z^{k-l+1} (k-l) \prod_{i=l+1}^{k+1} \frac{1}{\sqrt{\tilde{R}_{i,k}(z,1)}}. 
\end{align}

Again we perform a distinction of cases starting with $k$ not being an element of the sequence $(N_j)_{j \in \mathbb{N}}$.

\medskip
\paragraph{\textbf{The case: $N_j < k < N_{j+1}$}}

Let $\tilde{\rho}_{k}(u)$ be the dominant singularity of $_{k-l}\tilde{H}_{k}(z,u)$, which we know comes from the $(j+1)$-th radicand $\tilde{R}_{j+1,k}(z,u)$. Obviously, $\tilde{\rho}_{k}(1)=\rho_{k}(1)$. Therefore we will again use the abbreviation $\rho_k:=\tilde{\rho}_k(1)$. 

From Proposition \ref{prop:exprjkaroundsing} we get the following expansions of the radicands for $u=1$ and $\epsilon \longrightarrow 0$ so that $\epsilon \in \mathbb{C} \setminus \mathbb{R}^{-}$:

\begin{itemize}
 	\item $\forall i < j+1 \ \text{(inner radicands)}: \tilde{R}_{i,k}(\rho_k-\epsilon,1) = \tilde{R}_{i,k}(\rho_k,1)+ \mathcal{O}(|\epsilon|)$,\\
 	
 	\item $\tilde{R}_{j+1,k}(\rho_k-\epsilon,1) =  \tilde{\gamma}_{j+1} \epsilon +\mathcal{O}(|\epsilon|^2)$, 	
 	with $\tilde{\gamma}_{j+1}=-\frac{\partial}{\partial z} \tilde{R}_{j+1,k}(\rho_k,1)$,\\
 	
 	\item $\forall i > j+1 \ \text{(outer radicands)}: \tilde{R}_{i,k}(\rho_k-\epsilon,1) =\tilde{a}_i+\tilde{b}_i\sqrt{\epsilon}+\mathcal{O}(|\epsilon|),$ \\
 	
with $\tilde{a}_{i+1}=1-4(k-i)\rho_k^2-2\rho_k+2\rho_{k}\sqrt{\tilde{a}_i}$, and $\tilde{b}_{i+1}=\frac{\tilde{b}_i \rho_{k}}{\sqrt{\tilde{a}_i}}$ for $j+2 \leq i \leq k$, where $\tilde{a}_{j+2}=1-4(k-j-1)\rho_{k}^2-2\rho_k$ and $\tilde{b}_{j+2}=2\rho_k\sqrt{\tilde{\gamma}_{j+1}}$.
\end{itemize}

Thus, we have

\begin{itemize}
 	\item $\forall i < j+1 \ \text{(inner radicands)}: \frac{1}{\sqrt{\tilde{R}_{i,k}(\rho_k-\epsilon,1)} }= \frac{1}{\sqrt{\tilde{R}_{i,k}(\rho_k,1)}}+ \mathcal{O}(|\epsilon|)$,\\
 	
\item $\frac{1}{\sqrt{\tilde{R}_{j+1,k}(\rho_1-\epsilon,1)}} =
 \frac{1}{\sqrt{\tilde{\gamma}_{j+1}}} 
 \epsilon^{-\frac{1}{2}} +\mathcal{O} (|\epsilon|^{ \frac{1}{2}} )$,\\
 	
\item $\forall i > j+1 \ \text{(outer radicands)}: \frac{1}{\sqrt{\tilde{R}_{i,k}(\rho_k-\epsilon,1)}} 
=\frac{1}{ \sqrt{\tilde{a}_i}}-\frac{\tilde{b}_i}
{2\sqrt{\tilde{a}_i^3}} \epsilon^{ \frac{1}{2}}+\mathcal{O}(|\epsilon|^{ \frac{3}{2}})$.
\end{itemize}

Now we have to perform a distinction of cases whether the De Bruijn level that we are focussing on is below the $(k-j)$-th level or not (\textit{i.e.}, whether $l$ is below $j$ or not).

\smallskip
\paragraph{\underline{First case}: $l>j$}
First let us remember that $l>j$ implies that the $u$ is inserted in a radicand that is located outside the $(j+1)$-th.
From (\ref{equ:ablkminuslhleqk}) we get for $\epsilon \longrightarrow 0$ so that $\epsilon \in \mathbb{C} \setminus \mathbb{R}^{-}$
\[ \left( \frac{\partial}{\partial u} \ _{k-l}\tilde{H}_{k}(\rho_k-\epsilon,u) \right) \bigg|_{u=1} = \]
\[ \rho_k^{k-l+1} (k-l) \( \prod_{i=l+1}^{k+1} \frac{1}{\sqrt{\tilde{a}_i}} -   \sum_{m=l+1}^{k+1} \( \frac{\tilde{b}_m}{2 \sqrt{\tilde{a}_m^3}} \prod_{i=l+1, i \neq m}^{k+1} \frac{1}{\sqrt{\tilde{a}_i}} \) \epsilon^{\frac{1}{2}} + \mathcal{O}(|\epsilon|^{\frac{3}{2}}) \) . \]

By denoting the sum in the equation above with $\tilde{\delta}_l$ we can determine the coefficient of $z^n$ by
\begin{align*}
[z^n] \left( \frac{\partial}{\partial u} \ _{k-l}\tilde{H}_{k}(z,u) \right) \bigg|_{u=1} = 
- \rho_k^{k-l+1} (k-l) \tilde{\delta}_l \left( \frac{1}{\rho_k} \right)^n \frac{n^{-\frac{3}{2}}}{\Gamma(-\frac{1}{2})} \left(1+ \mathcal{O} \( \frac{1}{n} \) \right), \ \ \ \ \ \ \text{as} \ n \longrightarrow \infty,
\end{align*}

and by using the asymptotics of the $n$-th coefficient of  $_{k-l}\tilde{H}_{k}(z,1)=H_{k}(z,1)$ (see Theorem \ref{theo:asymphleqkueg1}) we finally get for the mean asymptotically as $n \longrightarrow \infty$
\[ \frac{[z^n] \left( \frac{\partial}{\partial u} \ _{k-l}\tilde{H}_{k}(z,u) \right) |_{u=1}}{[z^n]_{k-l}\tilde{H}_{k}(z,1)} = \frac{- \rho_k^{k-l+1} (k-l) \tilde{\delta}_l}{h_k} \( 1+ \mathcal{O} \( \frac{1}{n} \) \). \]

Thus, we showed that there is only a small number of leaves in the De Bruijn levels below the $(k-j)$-th level. More precisely, the asymptotic mean of the number of leaves is $O(1)$ for all these lower levels.

\medskip
\paragraph{\underline{Second case}: $l \leq j$}

Similar to the first case we get

\[ \left( \frac{\partial}{\partial u} \ _{k-l}\tilde{H}_{k}(\rho_k-\epsilon,u) \right) \bigg|_{u=1} = \]
\[
 \rho_k^{k-l+1} (k-l) \(  \( \prod_{i=l+1}^{j} \frac{1}{\sqrt{\tilde{R}_{i,k}(\rho_k,1)}} \) \( \prod_{i=j+2}^{k+1} \frac{1}{\sqrt{\tilde{a}_i}} \) \frac{1}{\sqrt{\tilde{\gamma}_{j+1}}}\epsilon^{-\frac{1}{2}} + \ \text{const. term} + \mathcal{O}(|\epsilon|^{\frac{1}{2}}) \).  \]

By setting $\tilde{\phi}_{j+1,l}:= \( \prod_{i=l+1}^{j} \frac{1}{\sqrt{\tilde{R}_{i,k}(\rho_k,1)}} \) \( \prod_{i=j+2}^{k+1} \frac{1}{\sqrt{\tilde{a}_i}} \) \frac{1}{\sqrt{\tilde{\gamma}_{j+1}}}$, we obtain for $n \longrightarrow \infty$

\[ [z^n] \left( \frac{\partial}{\partial u} \ _{k-l}\tilde{H}_{k}(z,u) \right) \bigg|_{u=1} =  \rho_k^{k-l+1} (k-l) \tilde{\phi}_{j+1,l} \( \frac{1}{\rho_k} \)^n \frac{n^{-\frac{1}{2}}}{\Gamma(\frac{1}{2})} \( 1+ \mathcal{O} \( \frac{1}{n} \) \). \]

Thus, we get for the mean asymptotically as $n \longrightarrow \infty$
\[ \frac{[z^n] \left( \frac{\partial}{\partial u} \ _{k-l}\tilde{H}_{k}(z,u) \right) |_{u=1}}{[z^n]_{k-l}\tilde{H}_{k}(z,1)} = \frac{\rho_k^{k-l+1} (k-l) \Gamma(-\frac{1}{2}) \tilde{\phi}_{j+1,l}}{\Gamma(\frac{1}{2})h_k} \cdot n \( 1+ \mathcal{O} \( \frac{1}{n} \) \). \]

Hence, we proved that the asymptotic mean for the number of leaves in the De Bruijn levels above the $(k-j)$-th is $\Theta(n)$. So, altogether we can see that almost all of the leaves are located in the upper $j+1$ De Bruijn levels. 

\medskip
\paragraph{\textbf{The case: $k = N_{j}$}}

Now we will deal with the second case, where the bound $k$ is an element of the sequence $(N_j)_{j \in \mathbb{N}}$.

We start by determining the expansions of the radicands around the dominant singularity $\tilde{\rho}_{k}(u)$ of $_{k-l}\tilde{H}_{k}(z,u)$ for $u=1$ and $\epsilon \longrightarrow 0$ so that $\epsilon \in \mathbb{C} \setminus \mathbb{R}^{-}$ (cf. \cite[Proposition 9]{bodini2015number}):

\medskip
\begin{itemize}
	\item $\forall i < j \ \text{(inner radicands)}: \tilde{R}_{i,k}(\rho_k-\epsilon,1) = \tilde{R}_{i,k}(\rho_k,1)+ \mathcal{O}(|\epsilon|), $\\
 	\item $\tilde{R}_{j,k}(\rho_k-\epsilon,1) =   \tilde{\gamma}_j \epsilon +\mathcal{O}(|\epsilon|^2) \ \ $ with $\tilde{\gamma}_j=-\frac{\partial}{\partial z}\tilde{R}_{j,k}(\rho_k,1)$,\\
 	\item $\tilde{R}_{j+1,k}(\rho_k-\epsilon,1) = 2 \tilde{\rho}_{k} \sqrt{\tilde{\gamma}_j} \epsilon^{\frac{1}{2}} + \mathcal{O}(|\epsilon|)$,\\	
 	\item $\forall i > j+1 \ \text{(outer radicands)}: \tilde{R}_{i,k}(\rho_k-\epsilon,1) = \tilde{a}_i+\tilde{b}_i\epsilon^{\frac{1}{4}}+\mathcal{O}(|\epsilon|),$ \\
 	
with $\tilde{a}_{i+1}=1-4(k-i)\rho_k^2-2\rho_k+2\rho_k\sqrt{\tilde{a}_i}$, and $\tilde{b}_{i+1}=\frac{\tilde{b}_i \rho_k}{\sqrt{\tilde{a}_i}}$ for $j+2 \leq i \leq k$, with $\tilde{a}_{j+2}=1-4(k-j)\rho_k^2-2\rho_k$ and $\tilde{b}_{j+2}=2\rho_k \sqrt{\tilde{\gamma}_{j}}$ .
\end{itemize}

Thus, we get

\begin{itemize}
 	\item $\forall i < j \ \text{(inner radicands)}: \frac{1}{\sqrt{\tilde{R}_{i,k}(\rho_k-\epsilon,1)} }= \frac{1}{\sqrt{\tilde{R}_{i,k}(\rho_k,1)}}+ \mathcal{O}(|\epsilon|)$,\\
 	
\item $\frac{1}{\sqrt{\tilde{R}_{j,k}(\rho_k-\epsilon,1)}} =
 \frac{1}{\sqrt{\tilde{\gamma}_j}} 
 \epsilon^{-\frac{1}{2}} +\mathcal{O} (|\epsilon|^{ \frac{1}{2}} )$,\\
 
 \item $\frac{1}{\sqrt{\tilde{R}_{j+1,k}(\tilde{\rho}_{k}-\epsilon,1)}} =
 \frac{1}{\sqrt{2\rho_k}\sqrt[4]{\tilde{\gamma}_j}} 
 \epsilon^{-\frac{1}{4}} +\mathcal{O} (|\epsilon|^{ \frac{1}{4}} )$,\\
 	
\item $\forall i > j+1 \ \text{(outer radicands)}: \frac{1}{\sqrt{\tilde{R}_{i,k}(\rho_k-\epsilon,1)}} 
=\frac{1}{ \sqrt{\tilde{a}_i}}-\frac{\tilde{b}_i}
{2\sqrt{\tilde{a}_i^3}} \epsilon^{ \frac{1}{4}}+\mathcal{O}(|\epsilon|)$. 
\end{itemize}

\medskip
We proceed analogously to the case where $N_j < k < N_{j+1}$, with the only difference that we have to distinguish between three cases now and since for $u=1$ the $j$-th and the $(j+1)$-th radicand vanish simultaneously, we get a closed formula for the dominant singularity $\rho_k=\frac{1}{1+\sqrt{1+4(k-j)}}$.

\medskip
\paragraph{\underline{First case}: $l>j$}
Let us again remember that $l>j$ implies that the $u$ is inserted in the $p$-th radicand with $p>j+1$.
From (\ref{equ:ablkminuslhleqk}) we get for $\epsilon \in \mathbb{C} \setminus \mathbb{R}^{-}$ with $|\epsilon| \longrightarrow 0$

\[ \left( \frac{\partial}{\partial u} \ _{k-l}\tilde{H}_{k}(\rho_k-\epsilon,u) \right) \bigg|_{u=1} = 
 \rho_k^{k-l+1} (k-l)  \prod_{i=l+1}^{k+1} \( \frac{1}{\sqrt{\tilde{a}_{i}}} -  \frac{\tilde{b}_{i}}{2 \sqrt{\tilde{a}_{i}^3}}  \epsilon^{\frac{1}{4}} + \mathcal{O}(|\epsilon|) \) =  \]
\[ \( \frac{1}{1+\sqrt{1+4(k-j)}} \)^{k-l+1} (k-l) \( \prod_{i=l+1}^{k+1} \frac{1}{\sqrt{\tilde{a}_{i}}} -   \sum_{m=l+1}^{k+1} \( \frac{\tilde{b}_{m}}{2 \sqrt{\tilde{a}_{m}^3}} \prod_{i=l+1, i \neq m}^{k+1} \frac{1}{\sqrt{\tilde{a}_{i}}} \)  \epsilon^{\frac{1}{4}} + \mathcal{O}(|\epsilon|^{\frac{1}{2}}) \) . \]

By setting $\tilde{\delta}_j := \sum_{m=l+1}^{k+1} \( \frac{\tilde{b}_{m}}{2 \sqrt{\tilde{a}_{m}^3}} \prod_{\substack{i=l+1 \\ i \neq m}}^{k+1} \frac{1}{\sqrt{\tilde{a}_{i}}} \)$, extracting the $n$-th coefficient and using the asymptotics of $[z^n]_{k-l}\tilde{H}_{k}(z,1)=[z^n]H_k(z,1)=\frac{-\rho_k^{k-j-1}b_{j+2}n^{-5/4}}{4\rho_k\Gamma(-1/4)\prod_{i=j+2}^{k+1}\sqrt{a_i}}\rho_k^{-n} \( 1+ \mathcal{O} \( \frac{1}{n} \) \)$ we have for $n \longrightarrow \infty$
\[ \frac{[z^n] \left( \frac{\partial}{\partial u} \ _{k-l}\tilde{H}_{k}(z,u) \right) |_{u=1}}{[z^n]_{k-l}\tilde{H}_{k}(z,1)} = \frac{-4\rho_k^{j-l+3} (k-l) \tilde{\delta}_j \prod_{m=j+2}^{k+1} \sqrt{a_m} }{b_{j+2}} \( 1 + \mathcal{O} \( n^{-\frac{1}{4}} \) \).\]

Thus, as in the previous case ($k \in (N_j, N_{j+1})$) the asymptotic mean for the number of leaves in the De Bruijn levels below the $(k-j)$-th level is $O(1)$.

\medskip
Furthermore the constant $D_{k,l}:=\frac{-4\rho_k^{j-l+3}(1) (k-l) \tilde{\delta}_j \prod_{m=j+2}^{k+1} \sqrt{a_m} }{b_{j+2}}$ can be simplified to
\begin{align}
\label{equ:Dlk}
D_{k,l}=\frac{k-l}{2\lambda_l} \( 1+ \frac{ \sqrt{\lambda_{l-j}}}{2\lambda_{l-j+1}} + \frac{ \sqrt{\lambda_{l-j}}}{4\lambda_{l-j+2}\sqrt{\lambda_{l-j+1}}} + \frac{ \sqrt{\lambda_{l-j}}}{8\lambda_{l-j+3}\sqrt{\lambda_{l-j+2}}\sqrt{\lambda_{l-j+1}}} + \ldots \),
\end{align}
with the sequence $\lambda_i$ defined by $\lambda_0=0$ and $\lambda_{i+1}=i+1+\sqrt{\lambda_i}$ for $i \geq 0$.


\medskip
\paragraph{\underline{Second case}: $l=j$}
Thus, the $u$ is inserted in the $(j+1)$-th radicand.
In this case we get

\begin{align}
\label{equ:meanNjleaveslevelj}
\frac{[z^n] \left( \frac{\partial}{\partial u} \ _{k-l}\tilde{H}_{k}(z,u) \right) |_{u=1}}{[z^n]_{k-l}\tilde{H}_{k}(z,1)} = \frac{-4\rho_k^{3} (k-j) \Gamma(-1/4) \psi_{j} \prod_{m=j+2}^{k+1} \sqrt{a_m}}{\Gamma(\frac{1}{4}) b_{j+2}} \cdot \sqrt{n} \ \( 1+ \mathcal{O} \( n^{-\frac{1}{4}} \) \), 
\end{align}
with
\begin{align} 
\label{equ:psij}
\psi_{j}= \frac{1}{\sqrt{2\rho_k} \sqrt[4]{\tilde{\gamma}_j}} \prod_{i=j+2}^{k+1} \frac{1}{\sqrt{a_{i}}}. 
\end{align}

The constant $\hat{D}_{k,l}:=\frac{-4\rho_k^{3} (k-j) \Gamma(-1/4) \psi_{j} \prod_{m=j+2}^{k+1} \sqrt{a_m}}{\Gamma(\frac{1}{4}) b_{j+2}}$ simplifies to
\begin{align*}
\hat{D}_{k,l}=\frac{- \Gamma(-1/4) (k-j) \sqrt{\rho_k}}{\Gamma(1/4) \sqrt{\tilde{\gamma}_j}}.
\end{align*}

In order to get some information on the magnitude of this factor we would have to investigate $\tilde{\gamma}_j=-\frac{\partial}{\partial z}\tilde{R}_{j,k}(\rho_k,1)$, which seems to get rather involved. However, taking a look at Equation (\ref{equ:meanNjleaveslevelj}) we can see that there are already considerably more unary nodes in the $(k-j)$-th De Bruijn level, namely $\Theta(\sqrt{n})$.

\medskip
\paragraph{\underline{Third case}: $l \leq j$}

The third case gives for $n \longrightarrow \infty$

\[ \frac{[z^n] \left( \frac{\partial}{\partial u} \ _{k-l}\tilde{H}_{k}(z,u) \right) |_{u=1}}{[z^n]_{k-l}\tilde{H}_{k}(z,1)} = \frac{-4\rho_k^{j-l+3}(k-l) \Gamma(-1/4) \chi_{j} \prod_{m=j+2}^{k+1} \sqrt{a_m}}{\Gamma(3/4) b_{j+2}} n \  \(1+ \mathcal{O} \( n^{-\frac{1}{4}} \) \), \]
with
\[ \chi_{j}= \frac{1}{\sqrt{\tilde{\gamma}_j}} \psi_j \( \prod_{i=l+1}^{j-1}  \frac{1}{\sqrt{\tilde{R}_{i,k}(\rho_k,1)}} \),\]

where $\psi_j$ is defined as in (\ref{equ:psij}).

Thus, we proved that asymptotically there is an average of $\Theta(n)$ leaves in the upper $j$ De Bruijn levels. The constant $\tilde{D}_{k,l}:=\frac{-4\rho_k^{j-l+3}(k-l) \Gamma(-1/4) \chi_{j} \prod_{m=j+2}^{k+1} \sqrt{a_m}}{\Gamma(3/4) b_{j+2}}$ can be rewritten as
\begin{align*}
\tilde{D}_{k,l}=\frac{-\Gamma(-1/4) (k-l) \rho_k^{j-l}}{\Gamma(3/4)\tilde{\gamma}_j} \prod_{i=l+1}^{j-1} \frac{1}{\sqrt{\tilde{R}_{i,k}(\rho_k,1)}}.
\end{align*}
The following proposition sums up all the results that we obtained within this section.

\begin{prop}
\label{prop:levelsmean}
Let $X_n$ denote the number of leaves in the $(k-l)$-th De Bruijn level in a random lambda-term of size $n$ with at most $k$ De Bruijn levels. 

If $k \in (N_j,N_{j+1})$, then we get for the asymptotic mean when $n \longrightarrow \infty$
\begin{itemize}
	\item in the case $l > j$: 
	\[ \mathbb{E}X_n=\frac{[z^n] \left( \frac{\partial}{\partial u} \ _{k-l}H_{k}(z,u) \right) |_{u=1}}{[z^n]_{k-l}H_{k}(z,1)} = C_{k,l} \( 1+ \mathcal{O} \( \frac{1}{n} \) \), \]

	\item and in the case $l \leq j$:
	\[ \mathbb{E}X_n=\frac{[z^n] \left( \frac{\partial}{\partial u} \ _{k-l}H_{k}(z,u) \right) |_{u=1}}{[z^n]_{k-l}H_{k}(z,1)} = \tilde{C}_{k,l} \cdot n \( 1+ \mathcal{O} \( \frac{1}{n} \) \), \] \end{itemize}
	with constants $C_{k,l}$ and $\tilde{C}_{k,l}$ depending on $l$ and $k$.
	
If $k =N_j$, then the asymptotic mean for $n \longrightarrow \infty$ reads as
\begin{itemize}
	\item in the case $l > j$: 
	\[ \mathbb{E}X_n=\frac{[z^n] \left( \frac{\partial}{\partial u} \ _{k-l}H_{k}(z,u) \right) |_{u=1}}{[z^n]_{k-l}H_{k}(z,1)} = D_{k,l} \( 1+ \mathcal{O} \( n^{-\frac{1}{4}} \) \), \]
	\item in the case $l = j$: 
	\[ \mathbb{E}X_n=\frac{[z^n] \left( \frac{\partial}{\partial u} \ _{k-l}H_{k}(z,u) \right) |_{u=1}}{[z^n]_{k-l}H_{k}(z,1)} = \hat{D}_{k,l} \cdot \sqrt{n} \( 1+ \mathcal{O} \( n^{-\frac{1}{4}} \) \), \]
	
	\item and in the case $l < j$:
	\[ \mathbb{E}X_n = \frac{[z^n] \left( \frac{\partial}{\partial u} \ _{k-l}H_{k}(z,u) \right) |_{u=1}}{[z^n]_{k-l}H_{k}(z,1)} = \tilde{D}_{k,l} \cdot n 
\( 1+ \mathcal{O} \( n^{-\frac{1}{4}} \) \), \] 

\end{itemize}
	with constants $D_{k,l}$, $\hat{D}_{k,l}$ and $\tilde{D}_{k,l}$ depending on $l$ and $k$.
\end{prop}

All the constants occurring in Proposition \ref{prop:levelsmean} have been calculated explicitly and can be obtained for every fixed $k$.
In particular, we investigated $D_{k,l}$ in order to show that for large $k$ the number of leaves in the De Bruijn levels that are closer to the root is smaller (\textit{cf}. Figure \ref{fig:summarykap5}).

\begin{prop}
Let us consider a random closed lambda-term of size $n$ with at most $k$ De Bruijn levels and let us consider the case $k=N_j$. Then the average number of leaves in De Bruijn level $L$, with $0 \leq L \leq k-j-1$, is asymptotically equal to a constant $C_L$, which behaves like
\begin{align*}
C_L \sim \frac{L}{2(k-j-1-L)} \ \ \ \ \text{as} \ k \longrightarrow \infty.
\end{align*}
\end{prop}

\begin{proof}
By setting $L:=k-l$ the constant $C_L$ corresponds exactly to the constant the $D_{k,l}$ of (\ref{equ:Dlk}). Thus, the proposition follows directly by investigating this constant $D_{k,l}$. The asymptotics for the sequence $\lambda_i$ ($\lambda_i \sim i, \text{as} \ i \longrightarrow \infty)$ can be obtained by bootstrapping.
\end{proof}

\begin{rem}
Using some estimates for the $a_i's$ we can prove that the same behaviour is true for the constants $C_{k,l}$. Thus, in both cases, whether $k$ is an element of $(N_i)_{i >0}$ or not, a random closed lambda-term with at most $k$ De Bruijn levels has almost no leaves in its lowest levels if $k$ is large.
\end{rem}

\subsubsection{Distributions}

Now that we derived the mean values for the number of leaves in the different De Bruijn levels, we are interested in their distribution. Therefore we distinguish again between the cases of $k$ being an element of the sequence $(N_i)_{i \geq 0}$ or not.

\medskip
\paragraph{\textbf{The case: $N_j < k < N_{j+1}$}}

We know that the generating function $_{k-l}\tilde{H}_k(z,u)$ consists of $k+1$ nested radicals, where a $u$ is inserted in the $(l+1)$-th radicand counted from the innermost one. Additionally we know that for $N_j < k < N_{j+1}$ the dominant singularity $\tilde{\rho}_k(u)$ comes from the  $(j+1)$-th radicand. Therefore, for $l>j$ the function $\tilde{\rho}_k(u)$ is independent of $u$, which is the reason why we do not get a quasi-power in that case. Thus, for the first $k-j$ levels of the lambda-DAG (\textit{i.e.} the case $l>j$), where there are just a few leaves, we can not say something about the distribution of the leaves so far. It might be a degenerated distribution.

However, in case that $l \leq j$ (\textit{i.e.}, for the upper levels where there are a lot of leaves) we will use the Quasi-Power theorem to show that the number of leaves in the $(k-j)$-th until the $k$-th level is asymptotically normally distributed.

Analogously as we did in Section \ref{subsec:total_height_njnj+1} we can show that
\begin{align}
\label{equ:kminusjhleqkzufracz1}
\frac{[z^n] \ _{k-l}\tilde{H}_{k}(z,u)}{[z^n] \ _{k-l}\tilde{H}_{k}(z,1)} = \frac{\tilde{h}_k(u)}{h_k} \( \frac{\rho_{k}}{\tilde{\rho}_{k}(u)} \)^n \( 1+ \mathcal{O} \( \frac{1}{n} \) \). 
\end{align}

We can easily see that Equation (\ref{equ:kminusjhleqkzufracz1}) has the desired shape for the Quasi-Power Theorem. Hence, assuming that $\tilde{B}''(1)+\tilde{B}'(1)-\tilde{B}'(1)^2 \neq 0$, where $\tilde{B}(u)=\frac{\rho_{k}}{\tilde{\rho}_{k}(u)}$, the Quasi-Power Theorem can be applied, which proves that the number of leaves in a De Bruijn level that is above the $(k-j-1)$-th level is asymptotically normally distributed.

\medskip
\paragraph{\textbf{The case: $k = N_{j}$}}

As is the previous case we do not know the distribution of the number of leaves in the lowest $k-j$ De Bruijn levels (\textit{i.e.}, the levels $0$ to $k-j-1$), due to the fact that for these levels the function $\tilde{\rho}_k(u)$ does not depend on $u$. It might also be a degenerated distribution.

In Section \ref{subsection:k=njtotal} we showed that the dominant singularity comes from the $j$-th radicand when $u$ is in a neighbourhood of 1. Thus, for the case that $l=j$, where we insert a $u$ in the $j+1$-th radicand, the dominant singularity $\rho_k(u)$ does still do not depend on $u$. Therefore we also do not know the distribution of the leaves in the $(k-j)$-th De Bruijn level. It seems very unlikely that the number of leaves in this level will be asymptotically normally distributed, but further studies on this subject might be very interesting.

Now we are going to show that the number of leaves in the upper $j$ De Bruijn levels (\textit{i.e.}, from the $(k-j+1)$-th to the $k$-th level) is asymptotically normally distributed. In order to do so we proceed analogously as in Section \ref{subsection:k=njtotal} for the total number of leaves.
Therefore for $l<j$ we set again $z=\tilde{\rho}_k(u)(1+\frac{t}{n})$ and $u=1+\frac{s}{n}$ and obtain expansions that behave just as the ones in Lemma \ref{lem:exprjrj+1rho1}. The only differences that occur concern the constants and therefore do not alter our results for the normal distribution. 

Thus, Theorem \ref{theo:levelsmeandist} is proved. Figure \ref{fig:summarykap5} summarizes the results that we obtained in Section \ref{sec:leaveslevels} and illustrates a combinatorial interpretation of the occurring phenomena.

\begin{figure}[h]
  \centering
\scalebox{0.8}{\begin{tikzpicture}
    \draw (2,4) -- (2,3);
    \draw (1.5,2.5) to [curve through = {(1.6,2.55) (1.7,2.6)}] (2,3);
    \draw (2.5,2.5) to [curve through = {(2.4,2.55) (2.3,2.6)}] (2,3);
    \draw (0,0) -- (1.25,2.5);
    \draw (4,0) -- (2.75,2.5);
	\draw (0,0) -- (4,0);
	\draw (1.25,2.5) -- (2.75,2.5);
	\draw (0.75,1.5) -- (3.25,1.5);
	
	\draw[dotted] (0.8,1.5) -- (0.8,1.6);
	\draw[dotted] (0.9,1.5) -- (0.9,1.8);
	\draw[dotted] (1,1.5) -- (1,1.9);
	\draw[dotted] (1.1,1.5) -- (1.1,2.2);
	\draw[dotted] (1.2,1.5) -- (1.2,2.4);
	\draw[dotted] (1.3,1.5) -- (1.3,1.8);
	\draw[dotted] (1.3,2.2) -- (1.3,2.5);
	\draw[dotted] (1.4,1.5) -- (1.4,1.8);
	\draw[dotted] (1.4,2.2) -- (1.4,2.5);
	\draw[dotted] (1.5,1.5) -- (1.5,1.8);
	\draw[dotted] (1.5,2.2) -- (1.5,2.5);
	\draw[dotted] (1.6,1.5) -- (1.6,1.8);
	\draw[dotted] (1.6,2.2) -- (1.6,2.5);
	\draw[dotted] (1.7,1.5) -- (1.7,1.8);
	\draw[dotted] (1.7,2.2) -- (1.7,2.5);
	\draw[dotted] (1.8,1.5) -- (1.8,1.8);
	\draw[dotted] (1.8,2.2) -- (1.8,2.5);
	\draw[dotted] (1.9,1.5) -- (1.9,1.8);
	\draw[dotted] (1.9,2.2) -- (1.9,2.5);
	\draw[dotted] (2,1.5) -- (2,1.8);
	\draw[dotted] (2,2.2) -- (2,2.5);
	\draw[dotted] (2.1,1.5) -- (2.1,1.8);
	\draw[dotted] (2.1,2.2) -- (2.1,2.5);
	\draw[dotted] (2.2,1.5) -- (2.2,1.8);
	\draw[dotted] (2.2,2.2) -- (2.2,2.5);
	\draw[dotted] (2.3,1.5) -- (2.3,1.8);
	\draw[dotted] (2.3,2.2) -- (2.3,2.5);
	\draw[dotted] (2.4,1.5) -- (2.4,1.8);
	\draw[dotted] (2.4,2.2) -- (2.4,2.5);
	\draw[dotted] (2.5,1.5) -- (2.5,1.8);
	\draw[dotted] (2.5,2.2) -- (2.5,2.5);
	\draw[dotted] (2.6,1.5) -- (2.6,1.8);
	\draw[dotted] (2.6,2.2) -- (2.6,2.5);
	\draw[dotted] (2.7,1.5) -- (2.7,1.8);
	\draw[dotted] (2.7,2.2) -- (2.7,2.5);
	\draw[dotted] (2.8,1.5) -- (2.8,2.4);
	\draw[dotted] (2.9,1.5) -- (2.9,2.2);
	\draw[dotted] (3,1.5) -- (3,1.9);
	\draw[dotted] (3.1,1.5) -- (3.1,1.8);
	\draw[dotted] (3.2,1.5) -- (3.2,1.6);
	
	\draw (0.1,0) -- (0.1,0.2);
	\draw (0.2,0) -- (0.2,0.4);
	\draw (0.3,0) -- (0.3,0.6);
	\draw (0.4,0) -- (0.4,0.8);
	\draw (0.5,0) -- (0.5,1);
	\draw (0.6,0) -- (0.6,1.2);
	\draw (0.7,0) -- (0.7,1.4);
	\draw (0.8,0) -- (0.8,1.5);
	\draw (0.9,0) -- (0.9,1.5);
	\draw (1,0) -- (1,1.5);
	\draw (1.1,0) -- (1.1,1.5);
	\draw (1.2,0) -- (1.2,1.5);
	\draw (1.3,0) -- (1.3,1.5);
	\draw (1.4,0) -- (1.4,1.5);
	\draw (1.5,0) -- (1.5,1.5);
	\draw (1.6,0) -- (1.6,1.5);
	\draw (1.7,0) -- (1.7,1.5);
	\draw (1.8,0) -- (1.8,1.5);
	\draw (1.9,0) -- (1.9,1.5);
	\draw (2,0) -- (2,1.5);
	\draw (2.1,0) -- (2.1,1.5);
	\draw (2.2,0) -- (2.2,1.5);
	\draw (2.3,0) -- (2.3,1.5);
	\draw (2.4,0) -- (2.4,1.5);
	\draw (2.5,0) -- (2.5,1.5);
	\draw (2.6,0) -- (2.6,1.5);
	\draw (2.7,0) -- (2.7,1.5);
	\draw (2.8,0) -- (2.8,1.5);
	\draw (2.9,0) -- (2.9,1.5);
	\draw (3,0) -- (3,1.5);
	\draw (3.1,0) -- (3.1,1.5);
	\draw (3.2,0) -- (3.2,1.5);
	\draw (3.3,0) -- (3.3,1.4);
	\draw (3.4,0) -- (3.4,1.2);
	\draw (3.5,0) -- (3.5,1);
	\draw (3.6,0) -- (3.6,0.8);
	\draw (3.7,0) -- (3.7,0.6);
	\draw (3.8,0) -- (3.8,0.4);
	\draw (3.9,0) -- (3.9,0.2);
    	\coordinate[label=+90:$k-j$] (z) at (2,1.7);
    	\node at (0.5,3) {$k=N_j$};
    	\end{tikzpicture}}
    	\scalebox{0.8}{\begin{tikzpicture} 
	\draw (2,4) -- (2,3);
    \draw (1.5,2.5) to [curve through = {(1.6,2.55) (1.7,2.6)}] (2,3);
    \draw (2.5,2.5) to [curve through = {(2.4,2.55) (2.3,2.6)}] (2,3);
    \draw (0,0) -- (1.25,2.5);
    \draw (4,0) -- (2.75,2.5);
	\draw (0,0) -- (4,0);
	\draw (1.25,2.5) -- (2.75,2.5);
	\draw (0.75,1.5) -- (3.25,1.5);
	\draw (1.25,2.5) -- (2.75,2.5);
	\draw (0.75,1.5) -- (3.25,1.5);
	
	\draw (0.8,1.5) -- (0.8,1.6);
	\draw (0.9,1.5) -- (0.9,1.8);
	\draw (1,1.5) -- (1,1.9);
	\draw (1.1,1.5) -- (1.1,2.2);
	\draw (1.2,1.5) -- (1.2,2.4);
	\draw (1.3,1.5) -- (1.3,1.8);
	\draw (1.3,2.2) -- (1.3,2.5);
	\draw (1.4,1.5) -- (1.4,1.8);
	\draw (1.4,2.2) -- (1.4,2.5);
	\draw (1.5,1.5) -- (1.5,1.8);
	\draw (1.5,2.2) -- (1.5,2.5);
	\draw (1.6,1.5) -- (1.6,1.8);
	\draw (1.6,2.2) -- (1.6,2.5);
	\draw (1.7,1.5) -- (1.7,1.8);
	\draw (1.7,2.2) -- (1.7,2.5);
	\draw (1.8,1.5) -- (1.8,1.8);
	\draw (1.8,2.2) -- (1.8,2.5);
	\draw (1.9,1.5) -- (1.9,1.8);
	\draw (1.9,2.2) -- (1.9,2.5);
	\draw (2,1.5) -- (2,1.8);
	\draw (2,2.2) -- (2,2.5);
	\draw (2.1,1.5) -- (2.1,1.8);
	\draw (2.1,2.2) -- (2.1,2.5);
	\draw (2.2,1.5) -- (2.2,1.8);
	\draw (2.2,2.2) -- (2.2,2.5);
	\draw (2.3,1.5) -- (2.3,1.8);
	\draw (2.3,2.2) -- (2.3,2.5);
	\draw (2.4,1.5) -- (2.4,1.8);
	\draw (2.4,2.2) -- (2.4,2.5);
	\draw (2.5,1.5) -- (2.5,1.8);
	\draw (2.5,2.2) -- (2.5,2.5);
	\draw (2.6,1.5) -- (2.6,1.8);
	\draw (2.6,2.2) -- (2.6,2.5);
	\draw (2.7,1.5) -- (2.7,1.8);
	\draw (2.7,2.2) -- (2.7,2.5);
	\draw (2.8,1.5) -- (2.8,2.4);
	\draw (2.9,1.5) -- (2.9,2.2);
	\draw (3,1.5) -- (3,1.9);
	\draw (3.1,1.5) -- (3.1,1.8);
	\draw (3.2,1.5) -- (3.2,1.6);

	\draw (0.1,0) -- (0.1,0.2);
	\draw (0.2,0) -- (0.2,0.4);
	\draw (0.3,0) -- (0.3,0.6);
	\draw (0.4,0) -- (0.4,0.8);
	\draw (0.5,0) -- (0.5,1);
	\draw (0.6,0) -- (0.6,1.2);
	\draw (0.7,0) -- (0.7,1.4);
	\draw (0.8,0) -- (0.8,1.5);
	\draw (0.9,0) -- (0.9,1.5);
	\draw (1,0) -- (1,1.5);
	\draw (1.1,0) -- (1.1,1.5);
	\draw (1.2,0) -- (1.2,1.5);
	\draw (1.3,0) -- (1.3,1.5);
	\draw (1.4,0) -- (1.4,1.5);
	\draw (1.5,0) -- (1.5,1.5);
	\draw (1.6,0) -- (1.6,1.5);
	\draw (1.7,0) -- (1.7,1.5);
	\draw (1.8,0) -- (1.8,1.5);
	\draw (1.9,0) -- (1.9,1.5);
	\draw (2,0) -- (2,1.5);
	\draw (2.1,0) -- (2.1,1.5);
	\draw (2.2,0) -- (2.2,1.5);
	\draw (2.3,0) -- (2.3,1.5);
	\draw (2.4,0) -- (2.4,1.5);
	\draw (2.5,0) -- (2.5,1.5);
	\draw (2.6,0) -- (2.6,1.5);
	\draw (2.7,0) -- (2.7,1.5);
	\draw (2.8,0) -- (2.8,1.5);
	\draw (2.9,0) -- (2.9,1.5);
	\draw (3,0) -- (3,1.5);
	\draw (3.1,0) -- (3.1,1.5);
	\draw (3.2,0) -- (3.2,1.5);
	\draw (3.3,0) -- (3.3,1.4);
	\draw (3.4,0) -- (3.4,1.2);
	\draw (3.5,0) -- (3.5,1);
	\draw (3.6,0) -- (3.6,0.8);
	\draw (3.7,0) -- (3.7,0.6);
	\draw (3.8,0) -- (3.8,0.4);
	\draw (3.9,0) -- (3.9,0.2);

    	\coordinate[label=+90:$k-j$] (z) at (2,1.7);

    	\node at (0.1,3) {$N_j<k<N_{j+1}$};
	\end{tikzpicture}}
\scalebox{0.8}{\begin{tikzpicture} 
	\draw (2,4) -- (2,3);
    \draw (1.5,2.5) to [curve through = {(1.6,2.55) (1.7,2.6)}] (2,3);
    \draw (2.5,2.5) to [curve through = {(2.4,2.55) (2.3,2.6)}] (2,3);
    \draw (0,0) -- (1.25,2.5);
    \draw (4,0) -- (2.75,2.5);
	\draw (0,0) -- (4,0);
	\draw (1.25,2.5) -- (2.75,2.5);
	\draw (0.75,1.5) -- (3.25,1.5);
	\draw (1.25,2.5) -- (2.75,2.5);
	\draw (0.75,1.5) -- (3.25,1.5);
	
	\draw[dotted] (0.8,1.5) -- (0.8,1.6);
	\draw[dotted] (0.9,1.5) -- (0.9,1.8);
	\draw[dotted] (1,1.5) -- (1,1.9);
	\draw[dotted] (1.1,1.5) -- (1.1,2.2);
	\draw[dotted] (1.2,1.5) -- (1.2,2.4);
	\draw[dotted] (1.3,1.5) -- (1.3,1.8);
	\draw[dotted] (1.3,2.2) -- (1.3,2.5);
	\draw[dotted] (1.4,1.5) -- (1.4,1.8);
	\draw[dotted] (1.4,2.2) -- (1.4,2.5);
	\draw[dotted] (1.5,1.5) -- (1.5,1.8);
	\draw[dotted] (1.5,2.2) -- (1.5,2.5);
	\draw[dotted] (1.6,1.5) -- (1.6,1.8);
	\draw[dotted] (1.6,2.2) -- (1.6,2.5);
	\draw[dotted] (1.7,1.5) -- (1.7,1.8);
	\draw[dotted] (1.7,2.2) -- (1.7,2.5);
	\draw[dotted] (1.8,1.5) -- (1.8,1.8);
	\draw[dotted] (1.8,2.2) -- (1.8,2.5);
	\draw[dotted] (1.9,1.5) -- (1.9,1.8);
	\draw[dotted] (1.9,2.2) -- (1.9,2.5);
	\draw[dotted] (2,1.5) -- (2,1.8);
	\draw[dotted] (2,2.2) -- (2,2.5);
	\draw[dotted] (2.1,1.5) -- (2.1,1.8);
	\draw[dotted] (2.1,2.2) -- (2.1,2.5);
	\draw[dotted] (2.2,1.5) -- (2.2,1.8);
	\draw[dotted] (2.2,2.2) -- (2.2,2.5);
	\draw[dotted] (2.3,1.5) -- (2.3,1.8);
	\draw[dotted] (2.3,2.2) -- (2.3,2.5);
	\draw[dotted] (2.4,1.5) -- (2.4,1.8);
	\draw[dotted] (2.4,2.2) -- (2.4,2.5);
	\draw[dotted] (2.5,1.5) -- (2.5,1.8);
	\draw[dotted] (2.5,2.2) -- (2.5,2.5);
	\draw[dotted] (2.6,1.5) -- (2.6,1.8);
	\draw[dotted] (2.6,2.2) -- (2.6,2.5);
	\draw[dotted] (2.7,1.5) -- (2.7,1.8);
	\draw[dotted] (2.7,2.2) -- (2.7,2.5);
	\draw[dotted] (2.8,1.5) -- (2.8,2.4);
	\draw[dotted] (2.9,1.5) -- (2.9,2.2);
	\draw[dotted] (3,1.5) -- (3,1.9);
	\draw[dotted] (3.1,1.5) -- (3.1,1.8);
	\draw[dotted] (3.2,1.5) -- (3.2,1.6);

	\draw (0.1,0) -- (0.1,0.2);
	\draw (0.2,0) -- (0.2,0.4);
	\draw (0.3,0) -- (0.3,0.6);
	\draw (0.4,0) -- (0.4,0.8);
	\draw (0.5,0) -- (0.5,1);
	\draw (0.6,0) -- (0.6,1.2);
	\draw (0.7,0) -- (0.7,1.4);
	\draw (0.8,0) -- (0.8,1.5);
	\draw (0.9,0) -- (0.9,1.5);
	\draw (1,0) -- (1,1.5);
	\draw (1.1,0) -- (1.1,1.5);
	\draw (1.2,0) -- (1.2,1.5);
	\draw (1.3,0) -- (1.3,1.5);
	\draw (1.4,0) -- (1.4,1.5);
	\draw (1.5,0) -- (1.5,1.5);
	\draw (1.6,0) -- (1.6,1.5);
	\draw (1.7,0) -- (1.7,1.5);
	\draw (1.8,0) -- (1.8,1.5);
	\draw (1.9,0) -- (1.9,1.5);
	\draw (2,0) -- (2,1.5);
	\draw (2.1,0) -- (2.1,1.5);
	\draw (2.2,0) -- (2.2,1.5);
	\draw (2.3,0) -- (2.3,1.5);
	\draw (2.4,0) -- (2.4,1.5);
	\draw (2.5,0) -- (2.5,1.5);
	\draw (2.6,0) -- (2.6,1.5);
	\draw (2.7,0) -- (2.7,1.5);
	\draw (2.8,0) -- (2.8,1.5);
	\draw (2.9,0) -- (2.9,1.5);
	\draw (3,0) -- (3,1.5);
	\draw (3.1,0) -- (3.1,1.5);
	\draw (3.2,0) -- (3.2,1.5);
	\draw (3.3,0) -- (3.3,1.4);
	\draw (3.4,0) -- (3.4,1.2);
	\draw (3.5,0) -- (3.5,1);
	\draw (3.6,0) -- (3.6,0.8);
	\draw (3.7,0) -- (3.7,0.6);
	\draw (3.8,0) -- (3.8,0.4);
	\draw (3.9,0) -- (3.9,0.2);

    	\coordinate[label=+90:$k-j-1$] (z) at (2,1.7);
	\node at (0.1,3) {$k=N_{j+1}$};
	\end{tikzpicture}}
	\caption{(1) In the $(k-j)$-th Be Bruijn level ($l=j$) are considerably more leaves than in the lower levels, but still less leaves then in the levels above. (2) With growing $k$ the $(k-j)$-th Be Bruijn level gets filled with leaves, while the number of leaves in the next level below (\textit{i.e.}, the $(k-j-1)$-th) slowly increases. (3) As soon as $k$ reaches the next element of the sequence $(N_j)_{j \geq 0}$, namely $k=N_{j+1}$ the $(k-j-1)$-th De Bruijn level immediately contains considerably more leaves than the levels below.}       
\label{fig:summarykap5}
\end{figure}
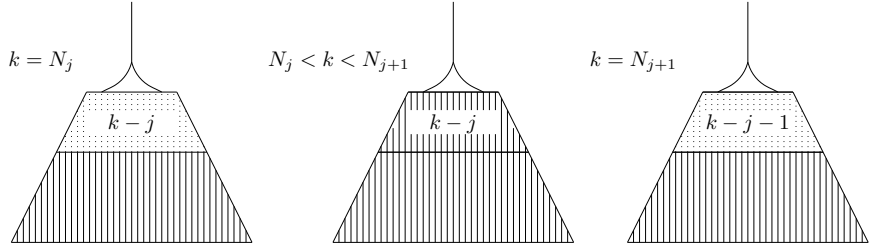

\subsection{Unary nodes}
\label{sec:unarylevels}

Now we want to investigate the number of unary nodes among the different De Bruijn levels. Let $C(z,u)$ again denote the generating function of the class of binary trees where $z$ marks the total number of nodes and $u$ marks the number of leaves, \textit{i.e.}, $C(z,u)=\frac{1-\sqrt{1-4z^2u}}{2z}$. The bivariate generating function $_{k-l}\bar{H}_k(z,w)$ of the class of closed lambda-terms with at most $k$ De Bruijn levels, where $z$ marks the size and $w$ the number of unary nodes in the $(k-l)$-th De Bruijn level, can then be expressed in terms of $C(z,u)$ by
\begin{align*}
C(z,C(z,1+\ldots + C(z,(k-l)+w\cdot C(z,\ldots (k-1)+C(z,k))\ldots)\ldots)).
\end{align*}
This can be rewritten to
\begin{align*}
_{k-l}\bar{H}_k(z,w) = \frac{1-\sqrt{\bar{R}_{k+1,k}(z,w)}}{2z},
\end{align*}
with
\begin{align*}
\bar{R}_{1,k}(z,w)&=1-4z^2k,\\
\bar{R}_{i,k}(z,w)&=1-4z^2(k-i+1)-2z+2z\sqrt{\bar{R}_{i-1,k}(z,w)} \ \ \ \text{for} \ 2 \leq i \leq k+1, \ \ i \neq l+1, \\
\bar{R}_{l+1,k}(z,w)&=1-4z^2(k-l)-2zw+2zw\sqrt{\bar{R}_{l,k}(z,w)}.
\end{align*}

Thus, for the derivatives we get
\begin{align*}
	\frac{\partial \bar{R}_{i,k}(z,w)}{\partial w} &= 0 \ \ \ \text{for} \ i < l+1, \\
	\frac{\partial \bar{R}_{l+1,k}(z,w)}{\partial w} &= -2z+2z\sqrt{\bar{R}_{l,k}(z,w)},\\
	\frac{\partial \bar{R}_{i,k}(z,w)}{\partial w} &= \frac{z}{\sqrt{\bar{R}_{i-1,k}}} \frac{\partial \bar{R}_{i-1,k}(z,w)}{\partial w} \ \ \ \text{for} \ i > l+1,
	\end{align*}
which implies
\begin{align*}
\frac{\partial _{k-l}\bar{H}_k(z,w)}{\partial w} \Big|_{w=1} = \frac{z^{k-l}}{2} \prod_{i=l+1}^{k+1} \frac{1}{\sqrt{\bar{R}_{i,k}}} \(1-\sqrt{\bar{R}_{l,k}}\).
\end{align*}

As in the previous section we distinguish between different cases.

\pagebreak
\subsubsection{The case: $k=N_j$}

\medskip
\paragraph{\underline{First case}: $l > j+1$}

Inserting the expansion of the radicands $\bar{R}_{i,k}$ and simplifying yields for $n \longrightarrow \infty$

\begin{align*}
[z^n]_{k-l}\bar{H}_{k}(z,w) \Big|_{w=1} = \rho_k^{k-l}{2} \alpha_l \frac{n^{-5/4}}{\Gamma(-1/4)} \rho_k^{-n} \( 1+ \mathcal{O} \(n^{-\frac{1}{4}} \) \),
\end{align*}
with
\begin{align}
\label{equ_alphal}
\alpha_l= -\frac{b_l}{2\sqrt{a_l}} \prod_{i=l+1}^{k+1} \frac{1}{\sqrt{a_i}} - \sum_{m=l+1}^{k+1} \frac{b_m}{2\sqrt{a_m^3}} \prod_{\substack{i=l+1 \\ i \neq m}}^{k+1} \frac{1}{\sqrt{a_i}} \( 1-\sqrt{a_l} \),
\end{align}
where $a_i:=\tilde{a}_i=a_i(1)$ and $b_i:=\tilde{b}_i=b_i(1)$ are defined in the previous sections and result from the expansions of the radicands. 

Thus, in this case the expected value of the number of unary nodes in the $(k-l)$-th De Bruijn level reads as
\begin{align*}
\frac{[z^n] \frac{\partial}{\partial w} _{k-l}\bar{H}_{k}(z,w)}{[z^n] _{k-l}\bar{H}_{k}(z,1)} = \frac{-2 \alpha_l \rho_k^{j-l+2} \prod_{m=j+2}^{k+1} \sqrt{a_m}}{b_{j+2}} \( 1+ \mathcal{O} \( n^{-\frac{1}{4}} \) \), \ \ \ \ \text{as} \ n \longrightarrow \infty.
\end{align*}

Furthermore, the constant $\frac{-2 \alpha_l \rho_k^{j-l+2} \prod_{m=j+2}^{k+1} \sqrt{a_m}}{b_{j+2}}$ can be simplified to
\begin{align*}
1+ \( \frac{1}{4\rho_k \lambda_{l-j}} - \frac{\sqrt{\lambda_{l-j-1}}}{2\lambda_{l-j}} \) \( 1+ \frac{\lambda_{l-j}}{2\lambda_{l-j+1}\sqrt{\lambda_{l-j}}} + \frac{\lambda_{l-j}}{2^2\lambda_{l-j+2} \sqrt{\lambda_{l-j+1}}\sqrt{\lambda_{l-j}}}+\ldots \),
\end{align*}
with the sequence $\lambda_i$ defined by $\lambda_0=0$ and $\lambda_{i+1}=i+1+\sqrt{\lambda_i}$ for $i \geq 0$.

Since the second summand is almost zero for $l$ being close to $k$ and large $k$, this implies that the number of unary nodes in these levels (close to the root) is close to one for large $k$.

\medskip
\paragraph{\underline{Second case}: $l = j+1$}

For $n \longrightarrow \infty$ we get 
\begin{align*}
[z^n]_{k-l}\bar{H}_{k}(z,w) \Big|_{w=1} = \rho_k^{k-l}{2} \zeta_l \frac{n^{-5/4}}{\Gamma(-1/4)} \rho_k^{-n} \( 1+ \mathcal{O} \(n^{-\frac{1}{4}} \) \),
\end{align*}
with
\begin{align*}
\zeta_l= -\sqrt{2\rho_k} (\rho_k \tilde{\gamma})^{1/4} \prod_{i=j+2}^{k+1} \frac{1}{\sqrt{a_i}} - \sum_{m=j+2}^{k+1} \frac{b_m}{2\sqrt{a_m^3}} \prod_{\substack{i=j+2\\ i \neq m}}^{k+1} \frac{1}{\sqrt{a_i}}.
\end{align*}

Thus,
\begin{align*}
\frac{[z^n] \frac{\partial}{\partial w} _{k-l}\bar{H}_{k}(z,w)}{[z^n] _{k-l}\bar{H}_{k}(z,1)} = \frac{-2 \zeta_l \rho_k^{j-l+2} \prod_{m=j+2}^{k+1} \sqrt{a_m}}{b_{j+2}} \( 1+ \mathcal{O} \( n^{-\frac{1}{4}} \) \), \ \ \ \ \text{as} \ n \longrightarrow \infty.
\end{align*}

In this case the constant $\frac{-2 \zeta_l \rho_k^{j-l+2} \prod_{m=j+2}^{k+1} \sqrt{a_m}}{b_{j+2}}$ simplifies to
\begin{align*}
1 + \frac{1}{4\rho_k} \( 1+ \frac{1}{2\lambda_2} + \frac{1}{2^2\lambda_3 \sqrt{\lambda_2}} + \ldots \).
\end{align*}

So, the expected number of unary nodes in the $(k-j-1)$-th De Bruijn level behaves exactly like in the lower levels. 
Starting from the next level a change in the behaviour can be determined, as we will see in the following.

\medskip
\paragraph{\underline{Third case}: $l = j$}

For $n \longrightarrow \infty$ we have
\begin{align*}
[z^n]_{k-l}\bar{H}_{k}(z,w) \Big|_{w=1} = \rho_k^{k-l}{2} \beta_l \frac{n^{-3/4}}{\Gamma(1/4)} \rho_k^{-n} \( 1+ \mathcal{O} \(n^{-\frac{1}{2}} \) \),
\end{align*}
with
\begin{align*}
\beta_l= \frac{1}{\sqrt{2\rho_k} \sqrt[4]{\rho_k \tilde{\gamma}_j}} \prod_{i=j+2}^{k+1} \frac{1}{\sqrt{a_i}}.
\end{align*}

Thus,
\begin{align*}
\frac{[z^n] \frac{\partial}{\partial w} _{k-l}\bar{H}_{k}(z,w)}{[z^n] _{k-l}\bar{H}_{k}(z,1)} = \frac{-2 \beta_l \rho_k^{2} \Gamma(-1/4) \prod_{m=j+2}^{k+1} \sqrt{a_m}}{\Gamma(1/4) b_{j+2}} \cdot \sqrt{n} \( 1+ \mathcal{O} \( n^{-\frac{1}{2}} \) \), \ \ \ \ \text{as} \ n \longrightarrow \infty.
\end{align*}

The constant $\frac{-2 \beta_l \rho_k^{2} \Gamma(-1/4) \prod_{m=j+2}^{k+1} \sqrt{a_m}}{\Gamma(1/4) b_{j+2}}$ can be written as
\begin{align*}
\frac{- \Gamma(-1/4)}{2\Gamma(1/4) \sqrt{\rho_k \tilde{\gamma}_j}}.
\end{align*}

The expected number of unary nodes in this ``separating level'' is therefore asymptotically $\Theta(\sqrt{n})$ (as was the number of leaves).

\medskip
\paragraph{\underline{Fourth case}: $l < j$}

For $ n \longrightarrow \infty$ we get
\begin{align*}
[z^n]_{k-l}\bar{H}_{k}(z,w) \Big|_{w=1} = \rho_k^{k-l}{2} \epsilon_l \frac{n^{-1/4}}{\Gamma(3/4)} \rho_k^{-n} \( 1+ \mathcal{O} \(n^{-\frac{1}{4}} \) \),
\end{align*}
with
\begin{align*}
\epsilon_l= \frac{1}{\sqrt{2\rho_k} \sqrt[4]{\rho_k \tilde{\gamma}_j} \sqrt{\rho_k \tilde{\gamma}_j}} \prod_{i=j+2}^{k+1} \frac{1}{\sqrt{a_i}} \prod_{m=l+1}^{j-1} \frac{1}{\sqrt{\tilde{R}_{m,k}}} \( 1- \sqrt{\bar{R}_{l,k}} \).
\end{align*}

Thus,
\begin{align*}
\frac{[z^n] \frac{\partial}{\partial w} _{k-l}\bar{H}_{k}(z,w)}{[z^n] _{k-l}\bar{H}_{k}(z,1)} = \frac{-2 \epsilon_l \rho_k^{j-l+2} \Gamma(-1/4) \prod_{m=j+2}^{k+1} \sqrt{a_m}}{\Gamma(3/4) b_{j+2}} \cdot n \( 1+ \mathcal{O} \( n^{-\frac{1}{4}} \) \), \ \ \ \ \text{as} \ n \longrightarrow \infty.
\end{align*}

The constant $\frac{-2 \epsilon_l \rho_k^{j-l+2} \Gamma(-1/4) \prod_{m=j+2}^{k+1} \sqrt{a_m}}{\Gamma(3/4) b_{j+2}}$ can be simplified to
\begin{align*}
\frac{-\rho_k^{j-l+1} \Gamma(-1/4)}{2\Gamma(3/4) \tilde{\gamma}_j} \prod_{m=l+1}^{j-1} \frac{1}{\sqrt{\tilde{R}_{m,k}}} \( 1- \sqrt{\bar{R}_{l,k}} \).
\end{align*}

Hence, analogously to the number of leaves, we proved that the number of unary nodes on the upper $j+1$ De Bruijn levels is $\Theta(n)$.

\subsubsection{The case: $N_j<k<N_{j+1}$}

This case works analogously to the previous one. Thus, we just give the results for the expected values.

\medskip
\paragraph{\underline{First case}: $l > j+1$}

In this case, the expected value is entirely equal to the mean for the case $k=N_j$ and $l > j+1$. So, with $\alpha_l$ defined as in (\ref{equ_alphal}), we have for $ n \longrightarrow \infty$

\begin{align*}
\frac{[z^n] \frac{\partial}{\partial w} _{k-l}\bar{H}_{k}(z,w)}{[z^n] _{k-l}\bar{H}_{k}(z,1)} = \frac{-2 \alpha_l \rho_k^{j-l+2}  \prod_{m=j+2}^{k+1} \sqrt{a_m}}{b_{j+2}} \cdot n \( 1+ \mathcal{O} \( n^{-\frac{1}{2}} \) \).
\end{align*}

\paragraph{\underline{Second case}: $l = j+1$}

In the second case, the constant differs a little bit, but the result stays qualitatively unaltered. We get

\begin{align*}
\frac{[z^n] \frac{\partial}{\partial w} _{k-l}\bar{H}_{k}(z,w)}{[z^n] _{k-l}\bar{H}_{k}(z,1)} = \frac{-2 \mu_l \rho_k^{j-l+2}  \prod_{m=j+2}^{k+1} \sqrt{a_m}}{b_{j+2}} \cdot n \( 1+ \mathcal{O} \( n^{-\frac{1}{2}} \) \), \ \ \ \ \text{as} \ n \longrightarrow \infty,
\end{align*}
with
\begin{align*}
\mu_l = - \sum_{m=j+2}^{k+1} \frac{b_m}{2\sqrt{a_m^3}} \prod_{\substack{ i=j+2 \\ i \neq m}}^{k+1} \frac{1}{\sqrt{a_i}} + \sqrt{\rho_k \tilde{\gamma}_{j+1}} \prod_{i=j+2}^{k+1} \frac{1}{\sqrt{a_i}}.
\end{align*}

\paragraph{\underline{Third case}: $l < j+1$}

For $n \longrightarrow \infty$ we have
\begin{align*}
\frac{[z^n] \frac{\partial}{\partial w} _{k-l}\bar{H}_{k}(z,w)}{[z^n] _{k-l}\bar{H}_{k}(z,1)} = \frac{-2 \theta_l \Gamma(-1/2) \rho_k^{j-l+2}  \prod_{m=j+2}^{k+1} \sqrt{a_m} \prod_{s=l+1}^{j} \sqrt{\bar{R}_s}}{b_{j+2} \Gamma(1/2)} \cdot n \( 1+ \mathcal{O} \( \frac{1}{n} \) \),
\end{align*}
with
\begin{align*}
\theta_l = \prod_{i=j+2}^{k+1} \frac{1}{\sqrt{a_i}} \frac{1}{\sqrt{\rho_k \tilde{\gamma}_{j+1}}} \(1- \sqrt{\bar{R}_{j,k}}\).
\end{align*}

Thus, the expected number of unary nodes in the last $j+1$ De Bruijn levels is asymptotically $\Theta(n)$.

\subsection{Binary nodes}
\label{sec:binary}

In this section we want to calculate the mean values of the number of binary nodes in the different De Bruijn levels. We denote by $C(z,v,u)$ the generating function of the class of binary trees where $z$ marks the total number of nodes, $v$ marks the number of binary nodes, and $u$ marks the number of leaves. Thus, we have
\begin{align}
\label{equ:gfbinarythree}
C(z,v,u)= \frac{1-\sqrt{1-4z^2uv}}{2zv}.
\end{align}

Using this generating function, we can write the bivariate generating function $_{k-l}H_k(z,v)$ of the class of closed lambda-terms with $z$ marking the size, and $v$ marking the the number of binary nodes on the $(k-l)$-th De Bruijn level as
\begin{align}
\label{equ:gfrekbinary}
C(z,1,C(z,1,1+C(z,1,2+\ldots +C(z,v,(k-l)+\ldots + C(z,1,k))\ldots)\ldots))).
\end{align}

Plugging in Equation (\ref{equ:gfbinarythree}) into (\ref{equ:gfrekbinary}) gives
\begin{align*}
_{k-l}\breve{H}_k(z,v) = \frac{1-\sqrt{\breve{R}_{k+1,k}(z,v)}}{2z},
\end{align*}
with
\begin{align*}
\breve{R}_{1,k}(z,v)&=1-4z^2k,\\
\breve{R}_{i,k}(z,v)&=1-4z^2(k-i+1)-2z+2z\sqrt{R_{i-1,k}(z,v)}, \ \ \ \text{for} \ 2 \leq i \leq k+1, \ \ i \neq l+1, \ i \neq l+2,\\
\breve{R}_{l+1,k}(z,v)&=1-4z^2(k-l)v-2zv+2zv\sqrt{\breve{R}_{l,k}(z,v)},\\
\breve{R}_{l+2,k}(z,v)&=1-4z^2(k-l-1)-\frac{2z}{v}+\frac{2z}{v}\sqrt{\breve{R}_{l+1,k}(z,v)}.
\end{align*}

Thus, for the derivatives we get
\begin{align*}
	\frac{\partial \breve{R}_{i,k}(z,v)}{\partial v} &= 0 \ \ \ \text{for} \ i < l+1, \\
	\frac{\partial \breve{R}_{l+1,k}(z,v)}{\partial v} &= -4z^2(k-l)-2z+2z\sqrt{\breve{R}_{l,k}(z,v)},\\
	\frac{\partial \breve{R}_{l+2,k}(z,v)}{\partial v} &= \frac{2z}{v^2} - \frac{2z}{v^2}\sqrt{\breve{R}_{l+1,k}}+ \frac{z}{v} \frac{1}{\sqrt{\breve{R}_{l+1,k}}} \frac{\partial \breve{R}_{l+1,k}(z,v)}{\partial v}, \\
	\frac{\partial \breve{R}_{i,k}(z,v)}{\partial v} &= z \frac{1}{\sqrt{\breve{R}_{i-1,k}}} \frac{\partial \breve{R}_{i-1,k}(z,v)}{\partial v} \ \ \ \text{for} \ i > l+2.
	\end{align*}

Finally, we have
\begin{align}
\frac{\partial _{k-l}\breve{H}_k(z,v)}{\partial v} \Big|_{v=1}=& 
\prod_{i=l+2}^{k+1} \frac{1}{\sqrt{\breve{R}_{i,k}}} \( - \frac{z^{k-l-1}}{2} +
\frac{\sqrt{\breve{R}_{l+1,k}}z^{k-l-1}}{2}+\frac{z^{k-l}}{2\sqrt{\breve{R}_{l+1,k}}}\right.
\nonumber \\
&\qquad     
\left.-\frac{z^{k-l}\sqrt{\breve{R}_{l,k}}}{2\sqrt{\breve{R}_{l+1,k}}}+\frac{z^{k-l+1}(k-l)}{\sqrt{\breve{R}_{l+1,k}}} \).
\label{equ:dHdvbinary}
\end{align}

Analogously to the previous sections we have to distinguish between different cases.
For the case $k=N_j$ and $l > j+1$ we get for $n \longrightarrow \infty$

\begin{align*}
[z^n]_{k-l}\breve{H}_{k}(z,v)= \xi_l \frac{n^{-5/4}}{\Gamma(-1/4)} \rho_k^{-n} \( 1+ \mathcal{O} \(n^{-\frac{1}{4}} \) \),
\end{align*}
with
\begin{align*}
\xi_l= &- \sum_{m=l+2}^{k+1} \frac{b_m}{2\sqrt{a_m^3}} \prod_{\substack{i=l+2 \\ i \neq m}} \frac{1}{\sqrt{a_i}} \( \frac{\rho_k^{k-l-1}(\sqrt{a_{l+1}}-1)}{2} + \frac{\rho_k^{k-l+1}(k-l)}{\sqrt{a_{l+1}}} + \frac{\rho_k^{k-l}(1-\sqrt{a_l})}{2\sqrt{a_{l+1}}} \) \\
&+ \prod_{i=l+2}^{k+1} \frac{1}{\sqrt{a_i}} \( \frac{\rho_k^{k-l-1}b_{l+1}}{2\sqrt{a_{l+1}}} - \frac{\rho_k^{k-l+1}(k-l)b_{l+1}}{2\sqrt{a_{l+1}^3}} + \frac{b_{l+1}}{2\sqrt{a_{l+1}^3}}(\sqrt{a_l}-1) - \frac{\rho_k^{k-l}b_l}{4\sqrt{a_l}\sqrt{a_{l+1}}} \).
\end{align*}
Thus,
\begin{align*}
\frac{[z^n] \frac{\partial}{\partial v} _{k-l}\breve{H}_{k}(z,v)}{[z^n] _{k-l}\breve{H}_{k}(z,1)} = \frac{-4 \xi_l \prod_{m=j+2}^{k+1} \sqrt{a_m}}{\rho_k^{k-j} b_{j+2}} \( 1+ \mathcal{O} \( n^{-\frac{1}{4}} \) \), \ \ \ \ \text{as} \ n \longrightarrow \infty.
\end{align*}

We performed a thorough investigation of the constant $ \frac{-4 \xi_l \prod_{m=j+2}^{k+1} \sqrt{a_m}}{\rho_k^{k-j} b_{j+2}}$ and showed that it is almost zero, in case $l$ is close to $k+1$ and $k$ is large,, \textit{i.e.}, if we consider a very low De Bruijn level, that is close to the root. 

Due to Equation (\ref{equ:dHdvbinary}) calculations get rather involved. Since the methods that are used are the same as in the previous section, we will omit further calculations. However, the results resemble the ones that we got in Section \ref{sec:leaveslevels} for the number of leaves. The only difference appears in the constants, but qualitatively also these constants behave equally.

\bibliographystyle{plain}
\bibliography{references}

\end{document}